\documentclass[aos,preprint]{imsart}
\setattribute{journal}{name}{}
\RequirePackage[OT1]{fontenc}
\RequirePackage{amsthm,amsmath}
\RequirePackage[numbers]{natbib}
\RequirePackage[colorlinks,citecolor=blue,urlcolor=blue]{hyperref}
\allowdisplaybreaks[4]
\usepackage{amssymb}
\usepackage{anysize}
\usepackage{hyperref}
\usepackage{extarrows}
\usepackage{multirow}

\usepackage{graphicx,psfrag,epsfig, subfigure}

\startlocaldefs
\numberwithin{equation}{section}
\newtheorem{thm}{Theorem}[section]
\newtheorem{lem}[thm]{Lemma}

\newtheorem{rem}[thm]{Remark}

\newtheorem{ass}[thm]{Assumption}

\newcommand{\be}{\begin{equation}}
\newcommand{\ee}{\end{equation}}
\newcommand{\bea}{\begin{eqnarray*}}
\newcommand{\eea}{\end{eqnarray*}}

\makeatletter

\newcommand{\Rmnum}[1]{\expandafter\@slowromancap\romannumeral #1@}
\makeatother
\makeatletter 
\@addtoreset{equation}{section}
\makeatother  

\endlocaldefs

\begin{document}

\begin{frontmatter}
\title{Canonical correlation coefficients of high-dimensional normal vectors: finite rank case}
\runtitle{High-dimensional canonical correlation}

\begin{aug}
\author{\fnms{Zhigang} \snm{Bao}\thanksref{t1}\ead[label=e1]{maomie2007@gmail.com}},
\author{\fnms{Jiang} \snm{Hu}\thanksref{t4}\ead[label=e4]{huj156@nenu.edu.cn}},
\author{\fnms{Guangming} \snm{Pan}\thanksref{t2}\ead[label=e2]{gmpan@ntu.edu.sg}}
\and
\author{\fnms{Wang} \snm{Zhou}\thanksref{t3}
\ead[label=e3]{stazw@nus.edu.sg}
\ead[label=u1,url]{http://www.sta.nus.edu.sg/~stazw/}}

\thankstext{t1}{Z.G. Bao was partially supported by the Ministry of Education, Singapore, under grant \# ARC 14/11}
\thankstext{t4}{J. Hu  was partially supported by CNSF 11301063}
\thankstext{t2}{G.M. Pan was partially supported by the Ministry of Education, Singapore, under grant \# ARC 14/11}
\thankstext{t3}{ W. Zhou was partially supported by the Ministry of Education, Singapore, under grant \# ARC 14/11,  and by a grant
R-155-000-116-112 at the National University of Singapore.}
\runauthor{Z.G. Bao et al.}

\affiliation{Nanyang Technological University\\ Northeast Normal
University \\and National University of
 Singapore}

\address{Division of Mathematical Sciences, \\School of Physical and Mathematical Sciences,\\ Nanyang Technological University, \\Singapore 637371\\
\printead{e1}\\
\phantom{E-mail:\ }}

\address{KLASMOE and School of Mathematics $\And$ Statistics\\ Northeast Normal
University \\P. R. China 130024\\
\printead{e4}\\
\phantom{E-mail:\ }}

\address{Division of Mathematical Sciences, \\School of Physical and Mathematical Sciences,\\ Nanyang Technological University, \\Singapore 637371\\
\printead{e2}\\
\phantom{E-mail:\ }}

\address{Department of Statistics and Applied Probability,\\ National University of
 Singapore,\\ Singapore 117546\\
\printead{e3}\\
\printead{u1}}
\end{aug}

\begin{abstract}
Consider a normal vector $\mathbf{z}=(\mathbf{x}',\mathbf{y}')'$, consisting of two sub-vectors $\mathbf{x}$ and $\mathbf{y}$ with dimensions $p$ and $q$ respectively. With $n$ independent observations of $\mathbf{z}$ at hand, we study the correlation between $\mathbf{x}$ and $\mathbf{y}$, from the perspective of the Canonical Correlation  Analysis, under the high-dimensional setting: both $p$ and $q$ are proportional to the sample size $n$. Denote by $\Sigma_{\mathbf{u}\mathbf{v}}$ (resp. $S_{\mathbf{u}\mathbf{v}}$) the population (resp. sample) cross-covariance matrix of any random vectors $\mathbf{u}$ and $\mathbf{v}$ (resp. their $n$ independent samples). The canonical correlation coefficients between $\mathbf{x}$ and $\mathbf{y}$ are known as the square roots of the nonzero eigenvalues of the canonical correlation matrix $\Sigma_{\mathbf{x}\mathbf{x}}^{-1}\Sigma_{\mathbf{x}\mathbf{y}}\Sigma_{\mathbf{y}\mathbf{y}}^{-1}\Sigma_{\mathbf{y}\mathbf{x}}$. In this paper, we focus on the case that $\Sigma_{\mathbf{x}\mathbf{y}}$ is of finite rank $k$, i.e. there are $k$ nonzero canonical correlation coefficients, whose squares are denoted by $r_1\geq\cdots\geq r_k>0$. Under the additional assumptions $(p+q)/n\to y\in (0,1)$ and $p/q\not\to 1$, we study the sample counterparts of $r_i,i=1,\ldots,k$, i.e. the largest k eigenvalues of the sample canonical correlation matrix $S_{\mathbf{x}\mathbf{x}}^{-1}S_{\mathbf{x}\mathbf{y}}S_{\mathbf{y}\mathbf{y}}^{-1}S_{\mathbf{y}\mathbf{x}}$, namely $\lambda_1\geq\cdots\geq \lambda_k$.  We show that there exists a threshold $r_c\in(0,1)$, such that for each $i\in\{1,\ldots,k\}$, when $r_i\leq r_c$, $\lambda_i$ converges almost surely to the right edge of the limiting spectral distribution of the sample canonical correlation matrix, denoted by $d_r$. When $r_i>r_c$, $\lambda_i$ possesses an almost sure limit in $(d_r,1]$, from which we can recover $r_i$ in turn, thus provide an estimate of the latter in the high-dimensional scenario.
\end{abstract}

\begin{keyword}[class=MSC]
\kwd{62H20, 60B20,60F99}
\end{keyword}

\begin{keyword}
\kwd{Canonical correlation analysis}
\kwd{ Random Matrices}
\kwd{ MANOVA ensemble}
\kwd{ High-dimensional data}
\kwd{ finite rank perturbation}
\kwd{ largest eigenvalues}
\end{keyword}

\end{frontmatter}
\section{Introduction} 
In multivariate analysis,  the most general and favorable method to investigate the relationship between two sets of random variables, arranged into the random vectors $\mathbf{x}$ and $\mathbf{y}$ respectively, is the Canonical Correlation Analysis (CCA), which was raised in the seminal work of Hotelling \cite{Hotelling1936}. CCA is aimed at seeking vectors $\mathbf{a}=\mathbf{a}_1$ and $\mathbf{b}=\mathbf{b}_1$ to maximize the correlation coefficient
\begin{eqnarray*}
\rho\equiv\rho(\mathbf{a},\mathbf{b}):=\frac{\text{Cov}(\mathbf{a}'\mathbf{x},\mathbf{b}'\mathbf{y})}{\sqrt{\text{Var}(\mathbf{a}'\mathbf{x})}\cdot\sqrt{\text{Var}(\mathbf{b}'\mathbf{y})}}.
\end{eqnarray*}
Conventionally,  $\rho_1:=\rho(\mathbf{a}_1,\mathbf{b}_1)$ is called the {\emph{first canonical correlation coefficient}}. Having obtained the first $m$ canonical correlation coefficients $\rho_i,i=1,\ldots,m$ and the corresponding vector pairs $(\mathbf{a}_i,\mathbf{b}_i),i=1\ldots,m$, one can proceed to seek vectors $(\mathbf{a}_{m+1},\mathbf{b}_{m+1})$ maximizing $\rho$ subject to the constraint that $(\mathbf{a}_{m+1}'\mathbf{x},\mathbf{b}_{m+1}'\mathbf{y})$ is uncorrelated with $(\mathbf{a}'_i\mathbf{x},\mathbf{b}_{i}'\mathbf{y})$ for all $i=1,\ldots,m$. Analogously, we call $\rho_i$ the {\emph{$i$th canonical correlation coefficient}} if it is nonzero. Denoting by $\Sigma_{\mathbf{u}\mathbf{v}}$ the population cross-covariance matrix of any random vectors $\mathbf{u}$ and $\mathbf{v}$, it is well known that 
$r_i:=\rho_i^2$ is the $i$th largest eigenvalue of the so-called {\emph{ (population) canonical correlation matrix}} \[\Sigma_{\mathbf{x}\mathbf{x}}^{-1}\Sigma_{\mathbf{x}\mathbf{y}}\Sigma_{\mathbf{y}\mathbf{y}}^{-1}\Sigma_{\mathbf{y}\mathbf{x}}.\] Drawing independently $n$ observations of the vector $\mathbf{z}:=(\mathbf{x}',\mathbf{y}')'\sim N(\boldsymbol{\mu},\Sigma)$ with mean vector $\boldsymbol{\mu}$ and covariance matrix
\begin{eqnarray*}
\Sigma=\left(\begin{array}{cc}
\Sigma_{\mathbf{x}\mathbf{x}} &\Sigma_{\mathbf{x}\mathbf{y}}\\
\Sigma_{\mathbf{y}\mathbf{x}} &\Sigma_{\mathbf{y}\mathbf{y}}
\end{array}\right),
\end{eqnarray*} 
namely $\mathbf{z}_i=(\mathbf{x}_i',\mathbf{y}_i')', i=1,\ldots,n$, we can study the canonical correlation coefficients via their sample counterparts. To be specific, we employ the notation $S_{\mathbf{u}\mathbf{v}}$ to represent the sample cross-covariance matrix for any two random vectors $\mathbf{u}$ and $\mathbf{v}$, where the implicit sample size of $(\mathbf{u}',\mathbf{v}')'$ is assumed to be $n$, henceforth. Then the square of the $i$th sample canonical correlation coefficient is defined as the $i$th largest eigenvalue of  the {\emph{sample canonical correlation matrix}} \[S_{\mathbf{x}\mathbf{x}}^{-1}S_{\mathbf{x}\mathbf{y}}S_{\mathbf{y}\mathbf{y}}^{-1}S_{\mathbf{y}\mathbf{x}},\] 
denoted by $\lambda_i$ in the sequel. 

Let $p$ and $q$ be the dimensions of the sub-vectors $\mathbf{x}$ and $\mathbf{y}$ respectively. In the classical low-dimensional setting, i.e., both $p$ and $q$ are fixed but $n$ is large, one can safely use $\lambda_i$ to estimate $r_i$, considering the convergence of the sample cross-covariance matrices towards their population counterparts, if $n$ is regarded to tend to $\infty$. However, nowadays, due to the increasing demand in the analysis of high-dimensional data springing up in various fields such as genomics, signal processing, microarray, finance and proteomics, putting forward a theory on high-dimensional CCA is much needed. In this paper, we will work with the following high-dimensional setting.
\begin{ass}[On the dimensions] \label{ass.070301} We assume that
$p:=p(n)$, $q:=q(n)$, and
\begin{eqnarray*}
p/n\to c_1\in(0,1),\quad q/n\to c_2\in (0,1),\quad q/p\not\to 1,\quad  \text{as}\quad n\to \infty,  \quad \text{s.t.} \quad c_1+c_2\in (0,1). 
\end{eqnarray*}
Without loss of generality, we always work with the additional assumption
\begin{eqnarray*}
p> q, \quad \text{thus} \quad c_1> c_2.
\end{eqnarray*}
\end{ass}
Let $\bar{\mathbf{x}}$ and $\bar{\mathbf{y}}$ be the sample means of $n$ samples $\{\mathbf{x}_i\}_{i=1}^n$ and $\{\mathbf{y}_i\}_{i=1}^n$ respectively, and use the notation $\mathring{\mathbf{x}}_i:=\mathbf{x}_i-\bar{\mathbf{x}}$ and $\mathring{\mathbf{y}}_i:=\mathbf{y}_i-\bar{\mathbf{y}}$ for $i=1,\ldots,n$.  We can then write
\begin{eqnarray*}
S_{\mathbf{x}\mathbf{x}}=\frac{1}{n-1}\sum_{i=1}^n\mathring{\mathbf{x}}_i\mathring{\mathbf{x}}_i',\quad S_{\mathbf{y}\mathbf{y}}=\frac{1}{n-1}\sum_{i=1}^n\mathring{\mathbf{y}}_i\mathring{\mathbf{y}}_i',\quad S_{\mathbf{x}\mathbf{y}}=\frac{1}{n-1}\sum_{i=1}^n\mathring{\mathbf{x}}_i\mathring{\mathbf{y}}_i',\quad S_{\mathbf{y}\mathbf{x}}=\frac{1}{n-1}\sum_{i=1}^n\mathring{\mathbf{y}}_i\mathring{\mathbf{x}}_i'.
\end{eqnarray*}
It is well known that there exist $n-1$ i.i.d. normal vectors $\tilde{\mathbf{z}}_i=(\tilde{\mathbf{x}}_i', \tilde{\mathbf{y}}_i')'
\sim N(\mathbf{0},\Sigma),
$
such that
\begin{eqnarray*}
S_{\mathbf{x}\mathbf{x}}=\frac{1}{n-1}\sum_{i=1}^{n-1}\tilde{\mathbf{x}}_i\tilde{\mathbf{x}}'_i,\quad S_{\mathbf{y}\mathbf{y}}=\frac{1}{n-1}\sum_{i=1}^{n-1}\tilde{\mathbf{y}}_i\tilde{\mathbf{y}}'_i,\quad S_{\mathbf{x}\mathbf{y}}=\frac{1}{n-1}\sum_{i=1}^{n-1}\tilde{\mathbf{x}}_i\tilde{\mathbf{y}}'_i,\quad S_{\mathbf{y}\mathbf{x}}=\frac{1}{n-1}\sum_{i=1}^{n-1}\tilde{\mathbf{y}}_i\tilde{\mathbf{x}}'_i.
\end{eqnarray*}
For simplicity, we recycle the notation $\mathbf{x}_i$ and $\mathbf{y}_i$ to replace $\tilde{\mathbf{x}}_i$ and $\tilde{\mathbf{y}}_i$,  and work with $n$ instead of $n-1$, noticing that such a replacement on sample size is harmless to Assumption \ref{ass.070301}. Hence, we can and do assume that $\mathbf{z}$ is centered in the sequel. By our assumption $p>q$, there are at most $q$ non-zero canonical correlations, either population ones or sample ones. An elementary fact is that $\lambda_{i},r_i\in[0,1]$ for all $i=1,\ldots, q$. Note that $\lambda_i,i=1,\ldots,q$ are also eigenvalues of the $q$ by $q$ matrix $S_{\mathbf{y}\mathbf{y}}^{-1}S_{\mathbf{y}\mathbf{x}}S_{\mathbf{x}\mathbf{x}}^{-1}S_{\mathbf{x}\mathbf{y}}$, whose empirical spectral distribution (ESD) will be denoted by 
\begin{eqnarray*}
F_n(x):=\frac{1}{q}\sum_{i=1}^q\mathbf{1}_{\{\lambda_i\leq x\}}.
\end{eqnarray*}
We will also write the data matrices 
\begin{eqnarray*}
X:=(\mathbf{x}_1,\ldots,\mathbf{x}_n),\quad Y:=(\mathbf{y}_1,\ldots, \mathbf{y}_n).
\end{eqnarray*}
Our aim, in this work, is to study the limit of $\lambda_i$ for any fixed positive integer $i$, when there is some fixed nonnegative integer $k$, such that
\begin{eqnarray*}
r_1\geq\ldots\geq r_k>r_{k+1}=\ldots=r_q=0.
\end{eqnarray*}
Formally, we make the following assumption throughout the work.
\begin{ass}[On the rank of the population matrix] \label{ass.070302} We assume that $\mathrm{rank}(\Sigma_{\mathbf{x}\mathbf{y}})=k$ for some fixed positive integer $k$. Furthermore, setting $r_0=1$, we denote by $k_0$ the nonnegative integer satisfying
\begin{eqnarray}1\geq\ldots \geq r_{k_0}> r_c\geq r_{k_{0}+1}\geq \ldots r_k> r_{k+1}=0, \label{062806}
\end{eqnarray}
where
\begin{eqnarray}
r_c\equiv r_c(c_1,c_2):=\frac{c_1c_2+\sqrt{c_1c_2(1-c_1)(1-c_2)}}{(1-c_1)(1-c_2)+\sqrt{c_1c_2(1-c_1)(1-c_2)}}. \label{062805}
\end{eqnarray}
\end{ass} In Section 1.2 we will state our main results. Before that, we introduce in Section 1.1 some known results in the {\emph{null case}}, i.e. $k=0$, which will be the starting point of our discussion.
\subsection{The null case: MANOVA ensemble} At first, we introduce some known results on the limiting behavior of $\{\lambda_i\}_{i=1}^q$ in the null case,  
i.e. $\mathbf{x}$ and $\mathbf{y}$ are independent, or else,  $r_i=0$ for all $i=1,\ldots,q$. It is elementary to see that the canonical correlation coefficients are invariant under the block diagonal transformation $(\mathbf{x}_i,\mathbf{y}_i)\to(A\mathbf{x}_i,B\mathbf{y}_i)$, for any $p\times p$ matrix $A$ and $q\times q$ matrix $B$, as long as both of them are nonsingular. For simplicity, in this section, we temporally assume that $\Sigma_{\mathbf{x}\mathbf{x}}=I_p$ and $\Sigma_{\mathbf{y}\mathbf{y}}=I_q$. Under our high-dimensional setting, i.e. Assumption \ref{ass.070301}, it is known that $\lambda_i$'s do not converge to $0$ even in the null case, instead, they typically spread out in an interval contained in $[0,1]$. Specifically, we have the following theorem on $F_n(x)$, essentially due to Wachter \cite{Wachter1980}.
\begin{thm} \label{thm.070301}When $\mathbf{x}$ and $\mathbf{y}$ are independent and Assumption \ref{ass.070301} holds, almost surely, $F_n$ converges weakly to  $F(x)$ possessing density
\begin{eqnarray*}
\rho(x)=\frac{1}{2\pi c_2}\frac{\sqrt{(d_r-x)(x-d_\ell)}}{x(1-x)}\mathbf{1}_{\{d_\ell\leq x\leq d_r\}},
\end{eqnarray*}
with
\begin{eqnarray*}
d_r=c_1+c_2-2c_1c_2+2\sqrt{c_1c_2(1-c_1)(1-c_2)},\quad d_\ell=c_1+c_2-2c_1c_2-2\sqrt{c_1c_2(1-c_1)(1-c_2)}.
\end{eqnarray*}
 \end{thm}
 \begin{rem} In the null case, the convergence of the ESD of the canonical correlation matrix actually holds under a more general distribution assumption as well, see \cite{YP2012}. 
 \end{rem}
 Conventionally, we call $F(x)$ in Theorem \ref{thm.070301} the limiting spectral distribution (LSD) of $S_{\mathbf{y}\mathbf{y}}^{-1}S_{\mathbf{y}\mathbf{x}}S_{\mathbf{x}\mathbf{x}}^{-1}S_{\mathbf{x}\mathbf{y}}$.
 One might note that $F(x)$ is just the LSD of the so-called MANOVA ensemble with appropriately chosen parameters, which is widely studied in the Random Matrix Theory (RMT). Actually, when $\mathbf{x}$ and $\mathbf{y}$ are normal and independent, the canonical correlation matrix $S_{\mathbf{y}\mathbf{y}}^{-1}S_{\mathbf{y}\mathbf{x}}S_{\mathbf{x}\mathbf{x}}^{-1}S_{\mathbf{x}\mathbf{y}}$ is exactly a MANOVA matrix, which can be seen by regarding $P_\mathbf{x}:=X'(XX')^{-1}X$ as a projection matrix independent of $Y$ thus
 \begin{eqnarray*}
S_{\mathbf{y}\mathbf{y}}^{-1}S_{\mathbf{y}\mathbf{x}} S_{\mathbf{x}\mathbf{x}}^{-1}S_{\mathbf{x}\mathbf{y}}=(Y(I-P_\mathbf{x})Y'+YP_\mathbf{x}Y')^{-1}YP_\mathbf{x}Y'
 \end{eqnarray*}
 and using Cochran's theorem to see that $Y(I-P_\mathbf{x})Y'$ and $YP_\mathbf{x}Y'$ are independent and 
 \begin{eqnarray*}
 Y(I-P_\mathbf{x})Y'\sim \text{Wishart}_q(I_q, n-p),\quad YP_\mathbf{x}Y'\sim\text{Wishart}_q(I_q, p).
 \end{eqnarray*}
Consequently, $\lambda_i,i=1,\ldots,q$ are known to possess the following joint density function,
 \begin{eqnarray*}
 p_n(\lambda_1,\ldots,\lambda_q)=C_n\prod_{i=1}^q(1-\lambda_i)^{(n-p-q-1)/2}\lambda_i^{(p-q-1)/2}\prod_{i<j}^q|\lambda_i-\lambda_j|,\quad \lambda_{i}\in [0,1], \quad i=1,\ldots,q,
 \end{eqnarray*}
 where $C_n$ is the normalizing constant, see Muirhead \cite{Muirhead1982}, page 112, for instance. Or else, one can refer to \cite{Johnstone2008}, for more related discussions.  In the language of RMT, after the change of variables $\lambda_i\to (1-\lambda_i)/2$, the point process possessing the above joint density turns out to be a Jacobi ensemble. 
 
 Throughout the paper, we will say that an $n$-dependent event $\mathbf{E}_n$ holds with {\emph{overwhelming probability}} if  for any positive number $\ell$, there exists
 \begin{eqnarray*}
 \mathbb{P}(\mathbf{E}_n)\geq 1-n^{-\ell}
 \end{eqnarray*}
 when $n$ is sufficiently large.
  The next crucial known result concerns the  convergence of the largest eigenvalues.
 \begin{thm} \label{thm.071501} When $\mathbf{x}$ and $\mathbf{y}$ are independent and Assumption \ref{ass.070301} holds, we have 
 \begin{eqnarray}
 \lambda_i\to d_r,\quad \text{a.s.} \label{071501}
 \end{eqnarray}
 for any fixed positive integer $i$. More precisely, for $\lambda_1$, one has that for any small constant $\eta>0$, 
 \begin{eqnarray}
\lambda_1\leq d_r+\eta \label{071502}
 \end{eqnarray}
 holds with overwhelming probability.
 \end{thm}
 \begin{rem} The estimate (\ref{071502}),  can actually be implied by some existing results in the literature directly. For example, one can refer to the small deviation estimate of the largest eigenvalue of the Jacobi ensemble in \cite{Katz2012} . Moreover, (\ref{071501}) is a direct consequence of (\ref{071502}) and Theorem \ref{thm.070301}. 
 \end{rem}
 \begin{rem} It is believed that, on the fluctuation level,  $\lambda_1$  possesses a Type 1 Tracy-Widom limit after appropriate normalization. Such a result has been partially established in \cite{Johnstone2008}.
 \end{rem}
\subsection{Finite rank case} We now turn to the case we are interested in: the finite rank case. To wit, Assumption \ref{ass.070302} holds. It will be clear that the sample canonical correlation matrix in such a finite rank case can be viewed as a finite rank perturbation of that in the null case. Consequently, the global behavior, especially the LSD, turns out to coincide with the null case. However, finite rank perturbation may significantly alter the behavior of the extreme eigenvalues, especially when the perturbation is strong enough. Similar problems have been studied widely for various random matrix models, not trying to be comprehensive, we refer to the spiked sample covariance matrices (\cite{BS2006, Paul2007, BY2008, BBP2005, FP2009}), the deformed Wigner matrices (\cite{CDF2009, peche2006, FP2009, CDF2012, KY2013}), the deformed unitarily invariant matrices (\cite{BBCF2012, Kargin2014}), and some more general models (\cite{BN2011, BGM2011}), either on the limit level or fluctuation level. In this work, for our sample canonical correlation matrix $S_{\mathbf{x}\mathbf{x}}^{-1}S_{\mathbf{x}\mathbf{y}}S_{\mathbf{y}\mathbf{y}}^{-1}S_{\mathbf{y}\mathbf{x}}$, we take the first step to study the limits of its largest eigenvalues, i.e. squares of the largest sample canonical correlation coefficients, under Assumption \ref{ass.070302}.
Our main result is the following theorem.
\begin{thm} \label{thm.061901} Under Assumptions \ref{ass.070301} and \ref{ass.070302}, the squares of the largest canonical correlation coefficients exhibit the following convergence as $n\to\infty$.
\begin{itemize}
\item[(i):] For $1\leq i\leq k_0$, we have
\begin{eqnarray*}
\lambda_i\stackrel{a.s.}\longrightarrow \gamma_i:=r_i(1-c_1+c_1r_i^{-1})(1-c_2+c_2r_i^{-1}).
\end{eqnarray*} 
\item[(ii):] For  each fixed $i\geq k_0+1$, we have
\begin{eqnarray*}
\lambda_i \stackrel{a.s.}\longrightarrow d_r.
\end{eqnarray*}
\end{itemize}
\end{thm}
The different limiting behavior of $\lambda_i$ in (i) and (ii) of Theorem \ref{thm.061901} can be observed in Figure \ref{fig1} below.
\begin{figure}[h]
\begin{center}
\includegraphics[width=10cm]{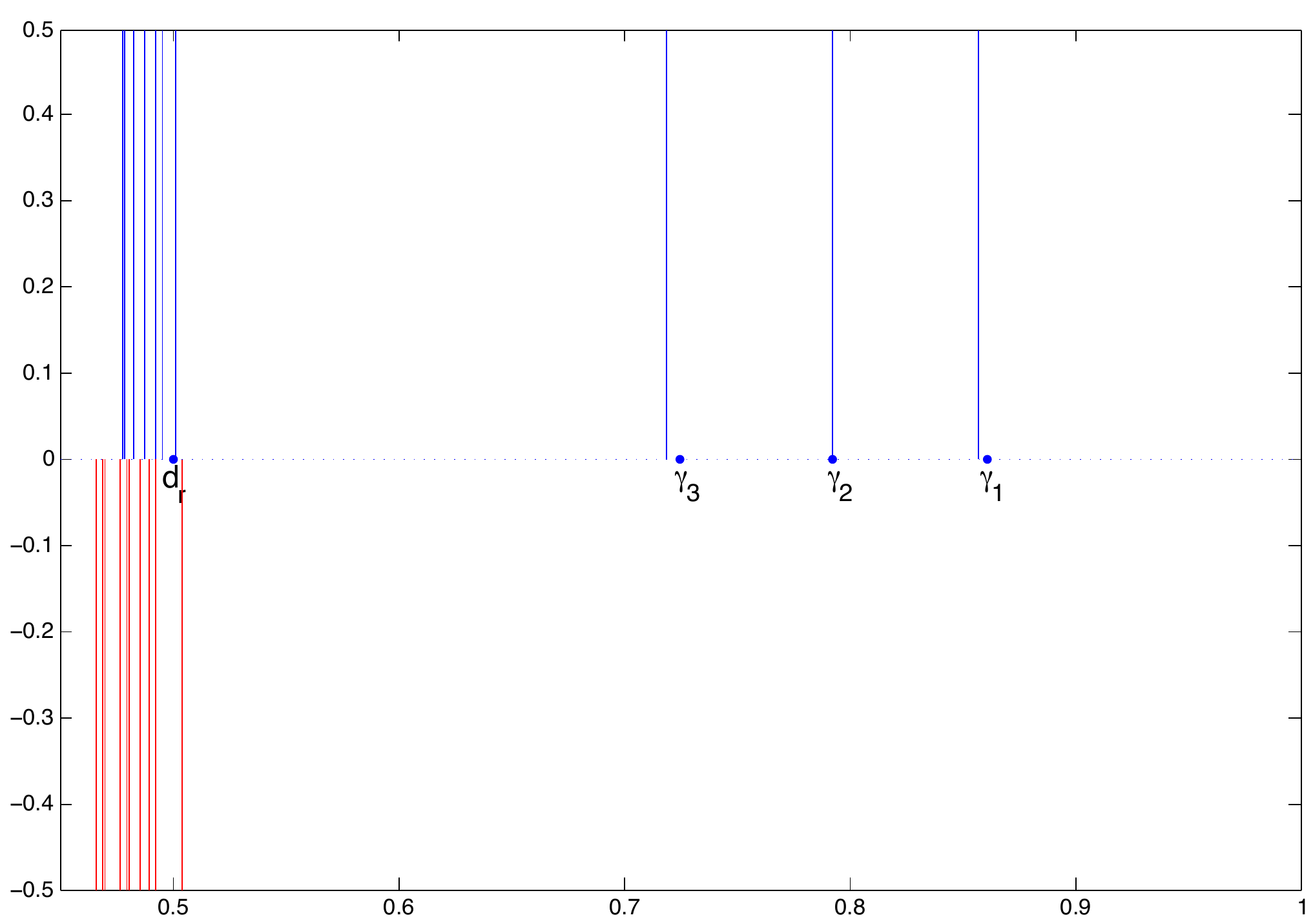}
\end{center}
\caption{The simulation was done in Matlab. We chose a normal vector $\mathbf{z}=(\mathbf{x}',\mathbf{y}')'$ with $p=500$ and $q=1000$. The sample size is $n=5000$. Hence, $c_1=0.1$, $c_2=0.2$. Then $r_c\approx 0.17$ and $d_r=0.5$. We chose $k=5$ and $(r_1, r_2, r_3, r_4, r_5)=(0.8, 0.7, 0.6, 0.16, 0.15)$. Then $\gamma_1\approx 0.86$, $\gamma_2\approx 0.79$, $\gamma_3\approx 0.73$. The abscises of the vertical segments represent the eigenvalues. The blue ones, above the dotted line, correspond to the largest $10$ eigenvalues of the sample canonical correlation matrices under the above setting, and the red ones, under the dotted line, correspond to the  largest $10$ eigenvalues in the null case.} \label{fig1}
\end{figure}
\subsection{Organization and notations}
Our paper is organized as follows. We introduce in Section 2 some necessary preliminaries. In Section 3, we will reformulate the sample canonical correlation matrix in the finite rank case as a perturbation of that in the null case, thereby obtaining a determinant equation for the largest eigenvalues, in the spirit of \cite{BN2011}. By solving the limiting equation of the determinant equation, we can get the limits of the largest eigenvalues in Section 4, i.e. prove Theorem \ref{thm.061901}, with some technical lemmas granted, especially those concerning the convergence of the determinant equation to the limiting one.  Sections 5, 6 and 7 are then devoted to proving these technical lemmas.

Throughout the paper,  $C$, $C_1$, $C_2$, $C_3$ and $C_4$ represent some generic constants whose values may vary from line to line. The notation $\mathbf{0}_{k_1\times k_2}$ is used to denote the $k_1$ by $k_2$ null matrix, which will be abbreviated to $\mathbf{0}_{k_1}$ if $k_1=k_2$. For any matrix $A$, its $(i,j)$th entry will be written as $A(i,j)$. When $A$ is square, we use $\text{Spec}(A)$ to denote its spectrum. For a function $f:\mathbb{C}\to\mathbb{C}$ and an Hermitian matrix $A$ with spectral decomposition $U_A\Lambda_AU^*_A$, we define $f(A)$ as usual, in the sense of functional calculus, to wit, $f(A)=U_Af(\Lambda_A)U^*_A$, where $f(\Lambda_A)$ is the diagonal matrix obtained via replacing the eigenvalues of $A$ by their images under $f$.
We will conventionally adopt the notation $||A||$ to represent the operator norm of a matrix $A$. While $||\mathbf{b}||$ stands for the Euclidean norm of $\mathbf{b}$ when it is a vector.
Let $\mathcal{P}=(\Omega, \mathcal{B}, \mu)$ be the underlying probability space of all random variables in this paper and set $\text{Mat}_m(\mathcal{P}):=\cap_{p=1}^\infty L^p(\mathcal{P}, \text{Mat}_m(\mathbb{C}))$, where $L^p(\mathcal{P}, \text{Mat}_m(\mathbb{C}))$ represents the set of all $m$ by $m$ random matrices whose entries are complex random variables on $\mathcal{P}$, possessing $p$th moments. For convenience, we say an $s$ by $t$ random matrix $S\sim N(\mathbf{0}, A\otimes B)$ if its vectorization follows a multivariate normal distribution with mean zero and covariance $A\otimes B$.
\section{Prelimenaries}
In this section, we introduce some basic notions and technical tools escorting our proofs and calculations in the subsequent sections. It turns out that in a key step, we need to provide the limits of the normalized trace of the matrices of the form $(A+B)^{-1}$, also known as the Stieltjes transform of the ESD of $A+B$ at origin when it is well-defined,  where $A$ and $B$ can be independent Wishart matrices or independent inverse-Wishart matrices, or their slight variants. As a sum of two independent orthogonally invariant matrices, its limiting global spectral property can usually be easily figured out by pursuing the tools from the Free Probability Theory, especially the commonly-used calculation scheme based on the Stieltjes transform and R-transform. We briefly review them below.
\begin{itemize}
\setlength{\itemindent}{-.2in}
\item {\emph{Stieltjes transform}}
\end{itemize}
For any given probability distribution $\sigma(\lambda)$, its Stieltjes transform is known as
\begin{eqnarray*}
s_{\sigma}(\omega):=\int\frac{1}{\lambda-\omega}d\sigma(\lambda),\quad \omega\in \mathbb{C}^+:=\{z\in\mathbb{C}: \Im z>0\}.
\end{eqnarray*}
From the definition, we can immediately get the fact that $\Im s_{\sigma}(\omega)>0$ for $\omega\in\mathbb{C}^+$. Actually,
the definition of $s_{\sigma}(\omega)$ can be extended to the domain $\mathbb{C}\setminus\text{supp}(\sigma)$ when $\sigma$ is compactly supported, by setting $s_\sigma(\bar{\omega})=\overline{s_\sigma(\omega)}$, where $\text{supp}(\sigma)$ represents the support of $\sigma(\lambda)$. Then $s_\sigma(\omega)$ is holomorphic on $\mathbb{C}\setminus \text{supp}(\sigma)$. More specifically, for our purpose, we focus on the Stieltjes transform of the spectral distribution of the sample covariance matrices in the sequel. For simplicity, we only state the result for a Wishart matrix $S\sim \text{Wishart}_p(I_p, n)$, as an example. For full generality, one can refer to the the monograph of Bai and Silverstein \cite{BS2009}.  
Let $F_{p,n}(\lambda)$ be the ESD of $n^{-1}S$. It is well known that, almost surely, $F_{p,n}$ converges weakly to its LSD, denoted by $F_{c_1}$, which follows the famous Marchenko-Pastur law (MP law), with density given by
\begin{eqnarray*}
\rho_{c_1}(\lambda)=\frac{1}{2\pi\lambda c_1}\sqrt{(d'_r-\lambda)(\lambda-d'_\ell)}\mathbf{1}_{\{d'_\ell\leq \lambda\leq d'_r\}},
\end{eqnarray*}
where $d'_r=(1+\sqrt{c_1})^2$ and $d'_\ell=(1-\sqrt{c_1})^2$. Now, correspondingly, we have the almost sure convergence of the Stieltjes transform. Denoting by $s_{p,n}(\omega)$ and $s_{c_1}(\omega)$ the Stieltjes transforms of $F_{p,n}$ and $F_{c_1}$ respectively, it is well known that $s_{p,n}(\omega)\to s_{c_1}(\omega)$ for all $\omega\in\mathbb{C}\setminus[d'_\ell,d'_r]$ almost surely. And we have 
\begin{eqnarray*}
s_{c_1}(\omega)=\frac{(1-c_1)-\omega+\sqrt{(\omega-1-c_1)^2-4c_1}}{2c_1\omega},
\end{eqnarray*}
where the square root is chosen to satisfy $\sqrt{(\omega-1-c_1)^2-4c_1}/\omega\to 1$ when $\omega\to\infty$.

\begin{itemize}
\setlength{\itemindent}{-.2in}
\item {\emph{R-transform}} 
\end{itemize}
 The R-transform for free additive convolution in the context of the Free Probability Theory, is just the analogue of the logarithm of the Fourier transform for the additive convolution in the classical probability theory. Here, we adopt an analysis-friendly 
definition of the R-transform in terms of the Stieltjes transform. 
For a probability distribution $\sigma(x)$, we define its Blue function $K_\sigma(\omega)$ as the formal functional inverse of $-s_\sigma(\omega)$, i.e. 
\begin{eqnarray}
-s_\sigma(K_\sigma(\omega))=K_\sigma(-s_\sigma(\omega))=\omega. \label{061802}
\end{eqnarray}
Then the R-transform of $\sigma(x)$ is defined by 
\begin{eqnarray}
R_\sigma(\omega)=K_\sigma(\omega)-\frac{1}{\omega} \label{061803}
\end{eqnarray}
which is holomorphic at $\omega=0$. 
Actually, there is
\begin{eqnarray}
R_\sigma(0)=\int xd\sigma(x) \label{062001}
\end{eqnarray}
when $\sigma(x)$ is compactly supported. A main result concerning the R-transform we need, due to Voiculescu \cite{Voiculescu1986}, is the well-known identity on free additive convolution, namely
\begin{eqnarray*}
R_{\sigma_\xi\boxplus\sigma_\zeta}=R_{\sigma_\xi}+R_{\sigma_\zeta},
\end{eqnarray*}
if $\xi$ and $\zeta$ are two free bounded operators in some noncommutative probability space $(\mathcal{A},\phi)$, where $\sigma_\xi$ and $\sigma_\zeta$ are the distributions of $\xi$ and $\zeta$ respectively, while $\sigma_\xi\boxplus\sigma_\zeta$  is the distribution of $a+b$. Here ``$+$" stands for the addition in the algebra $\mathcal{A}$. According to the seminal work of Voiculescu \cite{Voiculescu1991}, we know that in the noncommutative probability space $(\text{Mat}_{n}(\mathcal{P}), n^{-1}\mathbb{E}\text{tr})$, two independent $n$ by $n$ random matrices $A$ and $B$, both possessing LSDs, are asymptotically free if at least one of them is orthogonally invariant, thus the LSD of $A+B$ can be obtained via those of $A$ and $B$. To be specific, denoting by $R_{A+B}(\omega)$ the R-transform of the LSD of $A+B$, and by $R_A(\omega)$ and $R_B(\omega)$ those of $A$ and $B$ respectively, we have
\begin{eqnarray}
R_{A+B}(\omega)=R_A(\omega)+R_B(\omega). \label{061801}
\end{eqnarray}
Consequently, (\ref{061801}) provides us a tractable way to obtain the Stieltjes transform of the LSD of $A+B$ via those of $A$ and $B$, in light of the relation between the R-transform and the Stieltjes transform, to wit,  (\ref{061802}) and (\ref{061803}). 
\section{Determinant equation for the outliers} In this section, we derive a determinant equation for the outliers, i.e. the eigenvalues larger than $b_r$. Firstly, we will show that the canonical correlation matrix in the finite rank case can be viewed as a finite rank perturbation of that in the null case. To this end, technically, we will need the additional assumption
\begin{eqnarray*}
r_i\in(0,1), \quad \text{for all}\quad i=1,\ldots,k.
\end{eqnarray*}
 The case of  $r_i=1$ for at least one $i$ can be made up via a continuity argument in the proof of Theorem \ref{thm.061901}, presented in Section 4. 

\subsection{Reformulation: a perturbation of the null case}
As mentioned above, the canonical correlation coefficients are invariant under the block diagonal transformation $(\mathbf{x}_i,\mathbf{y}_i)\to(A\mathbf{x}_i,B\mathbf{y}_i)$, for any $p\times p$ matrix $A$ and $q\times q$ matrix $B$, as long as both of them are nonsingular. Hence, to study $\lambda_i$, we can start with the following setting
\begin{eqnarray*}
\left(
\begin{array}{cc}
X\\
Y
\end{array}
\right)=\Sigma^{1/2}
\left(
\begin{array}{cc}
W_1\\
W_2
\end{array}
\right),\quad \Sigma=\left(
\begin{array}{ccc}
I_p & R\\
R' & I_q
\end{array}
\right),
\end{eqnarray*}
where $W_1\sim N(\mathbf{0}, I_n\otimes I_p)$ and $W_2\sim  N(\mathbf{0}, I_n\otimes I_q)$ are independent, and \[R=\text{diag}(\sqrt{r_1},\ldots, \sqrt{r_k})\oplus \mathbf{0}_{(p-k)\times (q-k)},\] e.g. see Muirhead \cite{Muirhead1982}, page 530, formula (7). Embarking from the above setting, we need to take several steps of block diagonal transformations further, which will finally lead us to a perturbation formula. Setting 
\begin{eqnarray*}
\alpha_i:=\frac{\sqrt{1+\sqrt{r_i}}+\sqrt{1-\sqrt{r_i}}}{2},\quad \beta_i:=\frac{\sqrt{1+\sqrt{r_i}}-\sqrt{1-\sqrt{r_i}}}{2},
\end{eqnarray*}
we can obtain after elementary calculation that
\begin{eqnarray*}
\left(
\begin{array}{ccc}
I_p & R\\
R' & I_q
\end{array}
\right)^{1/2}
=\left(\begin{array}{ccc}
P_1 & P_3\\
P_3' & P_2
\end{array}
\right),
\end{eqnarray*}
where $P_1=\text{diag}(\alpha_1,\ldots,\alpha_k)\oplus I_{p-k}$, $P_2=\text{diag}(\alpha_1,\ldots,\alpha_k)\oplus I_{q-k}$, and $P_3=\text{diag}(\beta_1,\ldots,\beta_k)\oplus \mathbf{0}_{(p-k)\times (q-k)}$. Taking the transformation $(X,Y)\to (P^{-1}_1X, P^{-1}_2Y)$ yields the setting
\begin{eqnarray}
\left(
\begin{array}{cc}
X\\
Y
\end{array}
\right)=\left(\begin{array}{ccc}
I_p & P\\
P' & I_q
\end{array}
\right)
\left(
\begin{array}{cc}
W_1\\
W_2
\end{array}
\right),  \label{071510}
\end{eqnarray}
where
\begin{eqnarray*}
P=\text{diag}(\beta_1/\alpha_1,\ldots, \beta_k/\alpha_k)\oplus \mathbf{0}_{(p-k)\times (q-k)}=: \text{diag}(\tau_1,\ldots, \tau_k)\oplus \mathbf{0}_{(p-k)\times (q-k)}.
\end{eqnarray*}
Note that $\tau_i\neq 1$ in case $r_i\neq 1$, according to the definitions of $\alpha_i,\beta_i$ and $\tau_i$. Now, we perform a further transformation.
Denoting the matrix $Q=\text{diag}(2\tau_1/(1+\tau_1^2),\ldots,2\tau_k/(1+\tau_k^2))\oplus \mathbf{0}_{(p-k)\times (q-k)}$, and defining
\begin{eqnarray*}
W:=(I-QP')W_1+(P-Q)W_2,
\end{eqnarray*}
 it is easy to check that $W$ and $Y$ are independent, and
 \begin{eqnarray*}
 X=W+QY.
 \end{eqnarray*}
Moreover, observe that 
 \begin{eqnarray*}
 &&R_1:=n^{-1}\mathbb{E}WW'=I_p+\text{diag}\left(\frac{\tau_1^4-3\tau_1^2}{1+\tau_1^2},\ldots,\frac{\tau_k^4-3\tau_k^2}{1+\tau_k^2}\right)\oplus\mathbf{0}_{p-k},\nonumber\\
&& R_2:=n^{-1}\mathbb{E}YY'=I_q+\text{diag}(\tau_1^2,\ldots,\tau_k^2)\oplus\mathbf{0}_{q-k}.
 \end{eqnarray*}
Since $\tau_i\neq 1$ for all $i=1,\ldots,k$, it is clear that $R_1$ is nonsingular.
Relying on this fact, we  can introduce our final transformation $(X,Y)\to (\tilde{X},\tilde{Y})$ with
 \begin{eqnarray*}
\tilde{Y}:=R_2^{-1/2}Y,\quad \tilde{X}:=R_1^{-1/2}X=\tilde{W}+R_1^{-1/2}QR_2^{1/2}\tilde{Y},
 \end{eqnarray*}
 where $\tilde{W}:=R_1^{-1/2}W$. Consequently, we have
 \begin{eqnarray*}
\tilde{W}\sim N(\mathbf{0},I_n\otimes I_p),\quad \tilde{Y}\sim N(\mathbf{0},I_n\otimes I_q),\quad \tilde{X}\sim N(\mathbf{0},I_n\otimes (I_p+R_1^{-1/2}QR_2Q'R_1^{-1/2})).
 \end{eqnarray*}
By construction, we see that $\tilde{W}$ and $\tilde{Y}$ are independent and
 \begin{eqnarray*}
 R_1^{-1/2}QR_2^{1/2}=\text{diag}(\frac{2\tau_1}{1-\tau_1^2},\ldots,\frac{2\tau_k}{1-\tau_k^2})\oplus \mathbf{0}_{(p-k)\times(q-k)}.
 \end{eqnarray*}
 Denoting $t_i=2\tau_i/(1-\tau_i^2)$ for $i=1,\ldots,n$, we can easily get the relation
 \begin{eqnarray}
 r_i=\frac{t_i^2}{1+t_i^2}, \label{062804}
 \end{eqnarray}
according to the definitions of $\alpha_i$, $\beta_i$ and $\tau_i$ above.
 Hence, it suffices to study the limiting behaviour of $\lambda_i$'s in terms of $t_i$'s instead. For simplicity, we recycle the notation $X$, $W$ and $Y$ to denote $\tilde{X}$, $\tilde{W}$ and $\tilde{Y}$ respectively. In addition, we use the notation \[T:= R_1^{-1/2}QR_2^{1/2}=\text{diag}(t_1,\ldots,t_k)\oplus \mathbf{0}_{(p-k)\times (q-k)}\] in the sequel. Therefore, we can start with the following setting: 
\begin{eqnarray}
\left(\begin{array}{ccc}
X\\
Y
\end{array}
\right)=
\left(\begin{array}{ccc}
W+TY\\
Y
\end{array}
\right). \label{052203}
\end{eqnarray}
By the above construction, we have
\begin{eqnarray*}
\Sigma_{\mathbf{x}\mathbf{x}}=I_p+TT',\quad \Sigma_{\mathbf{y}\mathbf{y}}=I_q,\quad \Sigma_{\mathbf{x}\mathbf{y}}=T,\quad \Sigma_{\mathbf{y}\mathbf{x}}=T'.
\end{eqnarray*}
 Now, analogously, we introduce the notation
\begin{eqnarray*}
S_{\mathbf{w}\mathbf{w}}=n^{-1}WW',\quad S_{\mathbf{w}\mathbf{y}}=n^{-1}WY',\quad S_{\mathbf{y}\mathbf{w}}=n^{-1}YW' .
\end{eqnarray*}
In light of (\ref{052203}), we have
\begin{eqnarray}
&S_{\mathbf{x}\mathbf{x}}=S_{\mathbf{w}\mathbf{w}}+TS_{\mathbf{y}\mathbf{w}}+S_{\mathbf{w}\mathbf{y}}T'+TS_{\mathbf{y}\mathbf{y}}T',&\nonumber\\
&S_{\mathbf{x}\mathbf{y}}=S_{\mathbf{w}\mathbf{y}}+TS_{\mathbf{y}\mathbf{y}},\quad S_{\mathbf{y}\mathbf{x}}=S_{\mathbf{y}\mathbf{w}}+S_{\mathbf{y}\mathbf{y}}T'.& \label{061809}
\end{eqnarray}
Hence, due to the assumption that $\text{rank}(T)=k$, $S_{\mathbf{x}\mathbf{x}}^{-1}S_{\mathbf{x}\mathbf{y}}S_{\mathbf{y}\mathbf{y}}^{-1}S_{\mathbf{y}\mathbf{x}}$ can be obviously viewed as a finite rank perturbation of $S_{\mathbf{w}\mathbf{w}}^{-1}S_{\mathbf{w}\mathbf{y}}S_{\mathbf{y}\mathbf{y}}^{-1}S_{\mathbf{y}\mathbf{w}}$.  Based on this fact, we can derive the crucial determinant equation for those $\lambda_i$'s which are not in the spectrum of $S_{\mathbf{w}\mathbf{w}}^{-1}S_{\mathbf{w}\mathbf{y}}S_{\mathbf{y}\mathbf{y}}^{-1}S_{\mathbf{y}\mathbf{w}}$ in the next subsection.
\subsection{The determinant equation}  Note that, with probability $1$, the eigenvalues of $S_{\mathbf{x}\mathbf{x}}^{-1}S_{\mathbf{x}\mathbf{y}}S_{\mathbf{y}\mathbf{y}}^{-1}S_{\mathbf{y}\mathbf{x}}$  are the solutions for $\lambda$ of the characteristic equation
\begin{eqnarray}
\det(S_{\mathbf{x}\mathbf{y}}S_{\mathbf{y}\mathbf{y}}^{-1}S_{\mathbf{y}\mathbf{x}}-\lambda S_{\mathbf{x}\mathbf{x}})=0. \label{072501}
\end{eqnarray}
In light of (\ref{061809}), it is equivalent to
\begin{eqnarray*}
\det(S_{\mathbf{w}\mathbf{y}}S_{\mathbf{y}\mathbf{y}}^{-1}S_{\mathbf{y}\mathbf{w}}-\lambda S_{\mathbf{w}\mathbf{w}}+(1-\lambda)\Delta)=0,
\end{eqnarray*}
where
\begin{eqnarray*}
\Delta=TS_{\mathbf{y}\mathbf{w}}+S_{\mathbf{w}\mathbf{y}}T'+TS_{\mathbf{y}\mathbf{y}}T'.
\end{eqnarray*}
Then, if $\lambda$ is not an eigenvalue of $S_{\mathbf{w}\mathbf{w}}^{-1}S_{\mathbf{w}\mathbf{y}}S_{\mathbf{y}\mathbf{y}}^{-1}S_{\mathbf{y}\mathbf{w}}$ but  an eigenvalue of $S_{\mathbf{x}\mathbf{x}}^{-1}S_{\mathbf{x}\mathbf{y}}S_{\mathbf{y}\mathbf{y}}^{-1}S_{\mathbf{y}\mathbf{x}}$, it must satisfy the following equation 
\begin{eqnarray}
\det(I_p+(1-\lambda)\Phi(\lambda)\Delta)=0,  \label{061901}
\end{eqnarray}
where 
\begin{eqnarray*}
\Phi(\lambda)\equiv\Phi_n(\lambda):=(S_{\mathbf{w}\mathbf{y}}S_{\mathbf{y}\mathbf{y}}^{-1}S_{\mathbf{y}\mathbf{w}}-\lambda S_{\mathbf{w}\mathbf{w}})^{-1}.
\end{eqnarray*}
Denoting $\boldsymbol{\theta}_1=\mathbf{0}_{p\times (p-q)}$ and $\boldsymbol{\theta}_2=\mathbf{0}_{(p-q)\times (p-q)}$ simply, we can write
\begin{eqnarray*}
TS_{\mathbf{y}\mathbf{w}}+S_{\mathbf{w}\mathbf{y}}T'+TS_{\mathbf{y}\mathbf{y}}T'&=&
\left(
\begin{array}{ccc}
T &\boldsymbol{\theta}_1
\end{array}
\right)
\left(
\begin{array}{c}
S_{\mathbf{y}\mathbf{w}}\\
\boldsymbol{\theta}_1'
\end{array}
\right)+\left(
\begin{array}{ccc}
S_{\mathbf{w}\mathbf{y}} &\boldsymbol{\theta}_1
\end{array}
\right)
\left(
\begin{array}{c}
T'\\
\boldsymbol{\theta}_2
\end{array}
\right)\nonumber\\
&&+\left(
\begin{array}{ccc}
T &\boldsymbol{\theta}_1
\end{array}
\right)
\left(
\begin{array}{ccc}
S_{\mathbf{y}\mathbf{y}} &\boldsymbol{\theta}_1\\
\boldsymbol{\theta}_1' & \boldsymbol{\theta}_2
\end{array}
\right)
\left(
\begin{array}{c}
T'\\
\boldsymbol{\theta}_1'
\end{array}
\right)
\end{eqnarray*}
Let $\mathbf{e}_i$ be the $p$-dimensional vector with zero coefficients except the $i$th coefficient equal to $1$, and define the vector
\begin{eqnarray*}
\mathbf{u}_i:=\left(
\begin{array}{ccc}
S_{\mathbf{w}\mathbf{y}} &\boldsymbol{\theta}_1
\end{array}
\right)\mathbf{e}_i,\quad i=1,\ldots,k.
\end{eqnarray*}
Then we can write
\begin{eqnarray*}
TS_{\mathbf{y}\mathbf{w}}+S_{\mathbf{w}\mathbf{y}}T'+TS_{\mathbf{y}\mathbf{y}}T'=\sum_{i=1}^k t_i\mathbf{e}_i\mathbf{u}_i'+\sum_{i=1}^kt_i\mathbf{u}_i\mathbf{e}_i'+\sum_{i=1}^kt_it_jS_{\mathbf{y}\mathbf{y}}(i,j) \mathbf{e}_i\mathbf{e}_j'
\end{eqnarray*}
Let
\begin{eqnarray*}
\chi_{ij}:=t_it_jS_{\mathbf{y}\mathbf{y}}(i,j), \quad i,j=1,\ldots,k.
\end{eqnarray*}
We have
\begin{eqnarray*}
\Delta=\sum_{i=1}^k(\chi_{ii}\mathbf{e}_i\mathbf{e}_i'+t_i\mathbf{e}_i\mathbf{u}_i'+t_i\mathbf{u}_i\mathbf{e}_i')+\sum_{i\neq j}\chi_{ij}\mathbf{e}_i\mathbf{e}_j'.
\end{eqnarray*}
Now we further introduce the following matrices
\begin{eqnarray}
&&\mathcal{A}_i=(\chi_{ii}\mathbf{e}_i, t_i\mathbf{e}_i, t_i\mathbf{u}_i),\quad \mathcal{B}_i=(\mathbf{e}_i,\mathbf{u}_i,\mathbf{e}_i),\quad \mathcal{C}_i=(\underbrace{\mathbf{e}_i,\ldots, \mathbf{e}_i}_{k-1}),i=1,\ldots,k, \nonumber\\
&&\mathcal{F}_1=(\chi_{12}\mathbf{e}_2,\ldots, \chi_{1k}\mathbf{e}_k), \quad \mathcal{F}_i=(\chi_{i1}\mathbf{e}_1,\ldots, \chi_{i,i-1}\mathbf{e}_{i-1},\chi_{i,i+1}\mathbf{e}_{i+1},\ldots,\chi_{ik}\mathbf{e}_k),\quad i= 2,\ldots,k. \label{071603}
\end{eqnarray}
Defining the matrices
\begin{eqnarray}
U:=(\mathcal{A}_1,\ldots, \mathcal{A}_k, \mathcal{F}_1,\ldots, \mathcal{F}_k), \quad V:=(\mathcal{B}_1,\ldots, \mathcal{B}_k, \mathcal{C}_1,\ldots, \mathcal{C}_k)', \label{071604}
\end{eqnarray}
one has the factorization
\begin{eqnarray*}
\Delta=UV.
\end{eqnarray*}
Now, recalling (\ref{061901}) and using the well-known identity $\det(I+AB)=\det(I+BA)$ which is valid for any matrices $A$ and $B$ as long as both $AB$ and $BA$ are square, yields the determinant equation
\begin{eqnarray}
\det\left(I_{k^2+2k}+(1-\lambda)V\Phi(\lambda)U
\right)=0. \label{0604}
\end{eqnarray}
Hence, if we want to locate the eigenvalues outside the interval $[d_\ell,d_r]$, it suffices to solve (\ref{0604}) and find the limits of its solutions when $n\to\infty$. 
\section{Proof of Theorem \ref{thm.061901}} In this section, we provide the proof of Theorem \ref{thm.061901} based on several lemmas, whose proofs will be postponed to the subsequent sections. Our discussion consists of two thoroughly different parts, aimed at (i) and (ii) of Theorem \ref{thm.061901} respectively: 
(1) for the outliers, we locate them by deriving the limits of the solutions of the equation (\ref{0604}), which requires several steps of calculations, along with some reduction techniques; (2) for the eigenvalues sticking to $d_r$, we simply use Weyl's interlacing property to get the conclusion. 
\begin{itemize}
\item {\emph{The outliers}}
\end{itemize}
As claimed, to locate the outliers, we start from the equation (\ref{0604}). For simplicity, let
\begin{eqnarray}
M_n(z):=I_{k^2+2k}+(1-z)V\Phi(z)U \label{071801}
\end{eqnarray}
which is a well defined matrix-valued function for all $z\in \mathbb{C}\setminus \text{Spec}(S_{\mathbf{w}\mathbf{w}}^{-1}S_{\mathbf{w}\mathbf{y}}S_{\mathbf{y}\mathbf{y}}^{-1}S_{\mathbf{y}\mathbf{w}})$. 
Intuitively, if $M_n(z)$ is {\emph{close}} to some deterministic matrix-valued function $M(z)$ in some sense, it is reasonable to expect that the solutions of (\ref{0604}) are close to those of the equation $\det M(z)=0$. Such an implication can be explicitly formulated in the {\emph {location lemma}} given later, see Lemma \ref{lem.062902}. Before stating it, some notation should be introduced, in order to put forward our limiting target $M(z)$. Set for  any positive constant $\eta$ two domains
\begin{eqnarray*}
&&\mathcal{D}_1\equiv\mathcal{D}_1(\eta):=\{z\in\mathbb{C}: -\eta\leq \Re z\leq d_r+\eta, |\Im z|\leq \eta\},\\
&&\mathcal{D}_2\equiv\mathcal{D}_2(\eta):=\{z\in \mathbb{C}: d_r+\frac{3}{2}\eta< \Re z< 2, |\Im z|<1\}.
\end{eqnarray*}
Observe that $\mathcal{D}_2\subset\mathbb{C}\setminus\mathcal{D}_1$. Define the functions $\ell(z), h(z), f(z): \mathbb{C}\setminus\mathcal{D}_1\to \mathbb{C}$ as
\begin{eqnarray*}
&&\ell(z):=\sqrt{(z-d_\ell)(z-d_r)}=\sqrt{z^2+(4c_1c_2-2c_1-2c_2)z+(c_2-c_1)^2},\nonumber\\\\
&&h(z):=\frac{c_1+c_2-z+\ell(z)}{2c_2},\qquad f(z):=\frac{(2c_1-1)z+(c_2-c_1)+\ell(z)}{2c_1(1-c_1)z},
\end{eqnarray*}
where the square root  for $\ell(z)$ is taken to satisfy $\ell(z)/z\to1$ as $z\to\infty$. It is elementary to see that $\ell(z),h(z)$ and $f(z)$ are all holomorphic on $\mathbb{C}\setminus\mathcal{D}_1$. Now , let
\begin{eqnarray*}
G_i(z):=\left(
\begin{array}{ccc}
t_i^2f(z) &t_i f(z) & 0\\
0 &0 & t_ih(z)\\
t_i^2 f(z) & t_i f(z) &0
\end{array}
\right)
\end{eqnarray*}
and define
\begin{eqnarray*}
M(z):=I_{k^2+2k}+G_1(z)\oplus\cdots\oplus G_k(z)\oplus \mathbf{0}_{k^2-k}.
\end{eqnarray*}
Denoting for any given $\eta>0$ and $\varepsilon>0$ two events (in the underlying $\sigma$-algebra $\mathcal{B}$):
\begin{eqnarray*}
&&\Xi_1\equiv\Xi_1(n,\eta):=\bigg{\{} ||S_{\mathbf{w}\mathbf{w}}^{-1}S_{\mathbf{w}\mathbf{y}}S_{\mathbf{y}\mathbf{y}}^{-1}S_{\mathbf{y}\mathbf{w}}||\leq d_r+\frac{\eta}{2}\bigg{\}},\nonumber\\\\
&&\Xi_2\equiv\Xi_2(n,\eta,\varepsilon):=\bigg{\{}\sup_{z\in \mathcal{D}_2}\sup_{i,j=1,\ldots,k^2+2k}|M_n(z)(i,j)-M(z)(i,j)|\leq \varepsilon\bigg{\}},
\end{eqnarray*}
Note that, Theorem \ref{thm.071501} tells us that $\Xi_1$ holds with overwhelming probability.  Moreover, 
we emphasize here, in $\Xi_1$, $M_n(z)$ is obviously holomorphic on $\mathbb{C}\setminus \mathcal{D}_1$.  For $\Xi_2$, we have the following lemma.
\begin{lem} \label{lem.062802}
For any given $\eta>0$ and $\varepsilon>0$, $\Xi_2$ holds with overwhelming probability.
\end{lem}


The proof of Lemma \ref{lem.062802} is our main technical task, which will be postponed to the next three sections. The following  location lemma is a direct consequence of Lemma \ref{lem.062802}. At first, we remark here, it will be clear that the solutions of the equation $\det M(z)=0$ can only be real.
\begin{lem}[The location lemma] \label{lem.062902} For any given $\eta>0$,  let $z_1>\cdots>z_{s}$ be the solutions in $(d_r+\eta,1)$ of the equation $\det M(z)=0$, with multiplicities $m_1,\ldots, m_{s}$ respectively. Then for any fixed $\varepsilon>0$ and each $i\in\{1,\ldots,s\}$, with overwhelming probability, there exists $z_{n,i,1}>\cdots>z_{n,i,\tilde{s}_i}$ with multiplicities $m_{i,1},\ldots,m_{i,\tilde{s}_i}$ respectively, satisfying $\sum_{j=1}^{\tilde{s}_i}m_{i,j}=m_i$, such that
\begin{eqnarray}
\sup_{j=1\ldots, \tilde{s}_i}|z_{n,i, j}- z_i|\leq \varepsilon,  \label{062807}
\end{eqnarray}
and $\cup_{i=1}^{s}\{z_{n,i,j},j=1,\ldots, \tilde{s}_i\}$ is the collection of all solutions (multiplicities counted) of the equation $\det M_n(z)=0$ in $(d_r+\eta,1)$.
\end{lem}
\begin{proof}[Proof of Lemma \ref{lem.062902} with Lemma \ref{lem.062802} granted] At first, as mentioned above, in $\Xi_1$, $M_n(z)$ is holomorphic on $\mathbb{C}\setminus \mathcal{D}_1$. Moreover, by (i) of Lemma  \ref{lem.071001} below and the definition of $M_n(z)$ in (\ref{071801}), one sees that $||M_n(z)||$ is bounded uniformly on $\mathcal{D}_2$ with overwhelming probability, hence, the entries of $M_n(z)$ are all bounded in magnitude with overwhelming probability as well. In addition, it is clear that $M(z)$ is holomorphic and its entries are also bounded in magnitude on $\mathcal{D}_2$.
Hence, Lemma \ref{lem.062802} implies that with overwhelming probability,
\begin{eqnarray*}
\sup_{z\in \mathcal{D}_2}|\det M_n(z)-\det M(z)|\leq C\varepsilon
\end{eqnarray*}
for some positive constant $C$, taking into account the fact that the determinant is a multivariate polynomial of the matrix entries. It is obvious that $\det M_n(z)$  only has real roots since the equation (\ref{072501}) does. Then by using Rouche's theorem, we can immediately get (\ref{062807}).
\end{proof}
 Now, with the aid of Lemma \ref{lem.062902}, we prove (i) of Theorem \ref{thm.061901} in case $r_i<1$ for all $i\in\{1,\ldots,k\}$. Then, we extend the result to the case of $r_i=1$ for at least one $i\in\{1,\ldots,k\}$.
\begin{proof}[Proof of (i) of Theorem \ref{thm.061901} with $r_i$'s  less than $1$] According to Lemma \ref{lem.062902}, it suffices to solve $\det M(z)=0$ in $(d_r+\eta, 1)$ to get $z_i$, for sufficiently small $\eta>0$. By the definition of $M(z)$, we only need to solve the equation
\begin{eqnarray}
\det(I_3+G_i)=1+t_i^2f(z)-t_i^2f(z)h(z)=0. \label{071601}
\end{eqnarray}
It will be clear that there is a unique simple solution for the above equation. We denote it by $\gamma_i$ in the sequel.
Substituting the definitions of $f(z)$ and $h(z)$ into (\ref{071601}), we arrive at 
\begin{eqnarray}
\gamma_i-(c_1+c_2)-\ell(\gamma_i)=2c_1c_2 t_i^{-2}. \label{0623001}
\end{eqnarray}
According to the definition of $\ell(z)$, it suffices to find the solution of the equation 
\begin{eqnarray}
\gamma_i^2+(4c_1c_2-2c_1-2c_2)\gamma_i+(c_1-c_2)^2=(\gamma_i-c_1-c_2-2t^{-2}_ic_1c_2)^2 \label{0623002}
\end{eqnarray}
under the restriction that
\begin{eqnarray}
\gamma_i-(c_1+c_2)-2c_1c_2 t_i^{-2}\geq 0. \label{062901}
\end{eqnarray}
Solving (\ref{0623002}), we can get
\begin{eqnarray}
\gamma_i=\frac{(1+t^{-2}_ic_1)(1+t^{-2}_ic_2)}{1+t^{-2}_i}=r_i(1-c_1+c_1r_i^{-1})(1-c_2+c_2r_i^{-1}). \label{060901}
\end{eqnarray}
Note that
\begin{eqnarray}
\gamma_i&=&c_1(1-c_2)+c_2(1-c_1)+\frac{(1-c_1)(1-c_2)}{1+t^{-2}_i}+(1+t^{-2}_i)c_1c_2\nonumber\\
&&\geq c_1(1-c_2)+c_2(1-c_1)+2\sqrt{c_1c_2(1-c_1)(1-c_2)}=d_r \label{0607002}
\end{eqnarray}
Moreover, equality holds in the second step of (\ref{0607002})  only if
\begin{eqnarray}
(1+t^{-2}_i)^2=\frac{(1-c_1)(1-c_2)}{c_1c_2}\Longrightarrow t_i=\left[\frac{c_1c_2+\sqrt{c_1c_2(1-c_1)(1-c_2)}}{1-c_1-c_2}\right]^{1/2}:=t_c. \label{060801}
\end{eqnarray}
In addition, it is easy to check that when $t_i<t_c$, (\ref{062901}) fails. Hence, (\ref{0623001}) has solution only if $t_i\geq t_c$, with the solution $\gamma_i$ given by (\ref{060901}).  
Recalling the definition of $r_c$ in (\ref{062805}), It is elementary to see that
\begin{eqnarray*}
r_c=\frac{t_c^2}{1+t_c^2}.
\end{eqnarray*}
By the fact that $t\to t^2/(1+t^2)$ is an increasing function in $t\in [0,\infty)$, together with the relation (\ref{062804}) and the assumption (\ref{062806}), we see that\begin{eqnarray*}
r_i>r_c\Longrightarrow t_i>t_c\Longrightarrow \gamma_i>d_r,\quad i=1,\ldots,k_0.
\end{eqnarray*}
Hence, there exists some small $\eta>0$ such that  $\gamma_i> d_r+\eta$ for $i=1,\ldots, k_0$.
Now, what remains is to check $\gamma_i < 1$.
By (\ref{060901}), we see that $\gamma_i < 1$ is equivalent to
\begin{eqnarray*}
(1+t^{-2}_ic_1)(1+t^{-2}_ic_2)< 1+t^{-2}_i,
\end{eqnarray*}
which requires
\begin{eqnarray}
t_i> \sqrt{\frac{c_1c_2}{1-c_1-c_2}}.\label{0607003}
\end{eqnarray}
Note that $t_i>t_c$ automatically guarantees (\ref{0607003}). Then by (\ref{062807}) in Lemma \ref{lem.062802}, we get that 
\begin{eqnarray*}
\lambda_i\to \gamma_i, \quad \text{a.s.},\quad \text{for} \quad i=1,\ldots, k_0.
\end{eqnarray*}
Hence, we conclude the proof of (i) of Theorem  \ref{thm.061901} when $r_i<1$ for all $i=1,\ldots, k$.
\end{proof}
To extend the conclusion to the case that there is some $r_i=1$, we mention below the well known large deviation result of the extreme eigenvalues of Wishart matrices. Assume that $S\sim \text{Wishart}_{n_1}(I_{n_1}, n_2)$ for some $n_1:=n_1(n)$ and $n_2:=n_2(n)$ satisfying $n_1/n\to y_1\in(0,1)$, $n_2/n\to y_2\in(0,1)$  as $n$ tends to infinity, and $y:=y_1/y_2\in (0,1)$. Denoting $\lambda_1(n_2^{-1}S)$ and $\lambda_{n_1}(n_2^{-1}S)$ the largest and the smallest eigenvalues of $n_2^{-1}S$ respectively, it is well known that for any given positive number $\varepsilon>0$,
\begin{eqnarray}
(1+\sqrt{y})^2+\varepsilon\geq \lambda_1(n_2^{-1}S)\geq \lambda_{n_1}(n_2^{-1}S)\geq (1-\sqrt{y})^2-\varepsilon \label{0726100}
\end{eqnarray}
holds with overwhelming probability. For instance, one can refer to Theorem 5.9 of \cite{BS2006}, or Theorem 3.1 of \cite{PY2014}, for more delicate description, under much more general distribution assumption.
\begin{proof}[Proof of (i) of Theorem \ref{thm.061901} with $r_i=1$ for some $i$] We assume that there is an positive integer $k'\in \{1,\ldots, k_0\}$, such that $1=r_1=\ldots=r_{k'}>r_{k'+1}$. Note that, if $r_i=1$, there exists a pair of vectors $\boldsymbol{\mu}_i$ and $\boldsymbol{\nu}_i$ such that $X'\boldsymbol{\mu}_i=Y'\boldsymbol{\nu}_i$, which implies that
\begin{eqnarray*}
S_{\mathbf{xx}}^{-1}S_{\mathbf{xy}}S_{\mathbf{yy}}^{-1}S_{\mathbf{yx}}\boldsymbol{\mu}_i=\boldsymbol{\mu}_i,
\end{eqnarray*}
thus $\lambda_i=\gamma_i=r_i=1$, along with the corresponding eigenvector $\boldsymbol{\mu}_i$. Hence, $\lambda_i=1$ deterministically for $i=1,\ldots, k'$. However, such an observation does not tell us the convergence of $\lambda_i$ for any fixed $i\geq k'+1$. In the sequel, we use a continuity argument to make up this issue.
Note that in Section 3.1, up to (\ref{071510}), we do not need the assumption that all $r_i<1$. Hence, we can and do work with the setting (\ref{071510}). Recall the notation $P=\text{diag}(\tau_1,\ldots, \tau_k)\oplus \mathbf{0}_{(p-k)\times (q-k)}$.  Then by definition, one has $\tau_1=\cdots=\tau_{k'}=1$. We define for any fixed $\varepsilon>0$ a modification of $P$, namely $P^\varepsilon$ obtained by replacing each $\tau_1=\cdots=\tau_{k'}=1$ by $1-\varepsilon$. Correspondingly, we introduce the modified sample canonical correlation matrix $(S_{\mathbf{xx}}^\varepsilon)^{-1}S_{\mathbf{xy}}^\varepsilon(S_{\mathbf{yy}}^\varepsilon)^{-1}S_{\mathbf{yx}}^\varepsilon$, obtained from $S_{\mathbf{xx}}^{-1}S_{\mathbf{xy}}S_{\mathbf{yy}}^{-1}S_{\mathbf{yx}}$ via replacing $P$ by $P^\varepsilon$ in $X=W_1+PW_2$ and $Y=P'W_1+W_2$. Now, we claim that with overwhelming probability,
\begin{eqnarray}
\max\{||(S_{\mathbf{xx}}^\varepsilon)^{-1}-S_{\mathbf{xx}}^{-1}||, ||(S_{\mathbf{yy}}^\varepsilon)^{-1}-S_{\mathbf{yy}}^{-1}||, ||S_{\mathbf{xy}}^\varepsilon-S_{\mathbf{xy}}||\}\leq C\varepsilon \label{071805}
\end{eqnarray}
for some positive constant $C$. We verify the bound for $||(S_{\mathbf{xx}}^\varepsilon)^{-1}-S_{\mathbf{xx}}^{-1}||$ in the sequel, the other two can be handled analogously. Note that
\begin{eqnarray*}
||(S_{\mathbf{xx}}^\varepsilon)^{-1}-S_{\mathbf{xx}}^{-1}||\leq ||(S_{\mathbf{xx}}^\varepsilon)^{-1}||\cdot||S_{\mathbf{xx}}^{-1}||\cdot||S_{\mathbf{xx}}^\varepsilon-S_{\mathbf{xx}}||\leq C ||S_{\mathbf{xx}}^\varepsilon-S_{\mathbf{xx}}||
\end{eqnarray*}
with overwhelming probability, for some positive constant $C$. Here, in the last step,  we used the facts that $c_1<1$ and  both $S_{\mathbf{xx}}^\varepsilon$ and $S_{\mathbf{xx}}$ are Wishart matrices with positive-definite population covariance matrices, which imply via (\ref{0726100}) that the smallest eigenvalues of them are both bounded below by some positive constant with overwhelming probability. Then by definition, 
\begin{eqnarray*}
S_{\mathbf{xx}}^\varepsilon-S_{\mathbf{xx}}=\frac{1}{n}[(P^\varepsilon-P)W_2W_1'+W_1W_2'(P^\varepsilon-P)'+(P^\varepsilon W_2W_2' P^\varepsilon-P W_2W_2' P)],
\end{eqnarray*}
whose operator norm can be bounded by $C\varepsilon$ for some positive constant $C$ with overwhelming probability, taking into account the facts that $||P-P^\varepsilon||\leq \varepsilon$ and $||W_1/\sqrt{n}||, ||W_2/\sqrt{n}||\leq C$ for some positive constant $C$ with overwhelming probability, in light of (\ref{0726100}). Following from (\ref{071805}), it is easy to deduce that for some positive constant $C$,
\begin{eqnarray*}
||(S_{\mathbf{xx}}^\varepsilon)^{-1}S_{\mathbf{xy}}^\varepsilon(S_{\mathbf{yy}}^\varepsilon)^{-1}S_{\mathbf{yx}}^\varepsilon-S_{\mathbf{xx}}^{-1}S_{\mathbf{xy}}S_{\mathbf{yy}}^{-1}S_{\mathbf{yx}}||\leq C\varepsilon
\end{eqnarray*}
holds with overwhelming probability. Since $\varepsilon$ can be chosen to be arbitrarily small,  $\lambda_i$ is arbitrarily close to  $\gamma_i$ almost surely, for $i\in \{k'+1,k_0\}$. Hence, we prove the case when there is  $r_i=1$ for at least one $i\in\{1,\ldots, k\}$.
\end{proof}
Now, for (i) of Theorem \ref{thm.061901}, what remains is to prove Lemma \ref{lem.062802}. We perform it in the following sections. Before that, we state the proof for (ii) of Theorem \ref{thm.061901}.
\begin{itemize}
\item {\emph{The Sticking eigenvalues}} 
\end{itemize}
\begin{proof}[Proof of (ii) of Theorem \ref{thm.061901}]
At first, according to the proof of (i) of Theorem \ref{thm.061901}, we see that with overwhelming probability, there are exactly $k_0$ largest eigenvalues  (multiplicities counted) of $S_{\mathbf{xx}}^{-1}S_{\mathbf{xy}}S_{\mathbf{yy}}^{-1}S_{\mathbf{yx}}$ in the interval $(d_r+\eta,1]$, for any sufficiently small $\eta>0$. Hence, we see that $\limsup_{n\to\infty}\lambda_i\leq d_r+\eta$ for all $i\geq k_0+1$ almost surely. Moreover, by the Weyl's interlacing property and (\ref{071501}), we always have $\liminf_{n\to\infty}\lambda_i\geq d_r-\eta$ almost surely for any fixed $i$. Since $\eta$ can be arbitrarily small, we get the conclusion that $\lambda_i\stackrel{a.s.}\to d_r$ for any fixed $i\geq k_0+1$.  Hence, we complete the proof.
\end{proof}
\section{Modification and Reduction}
To prove $M_n(z)\to M(z)$ in the sense of Lemma \ref{lem.062802}, there will be several steps which can be simplified by introducing further slight modifications, to bring in deterministic boundedness of the matrix $M_n(z)$ on the region of interest. To this end, we introduce an entrywise modification of $M_n(z)$, resulting in a proxy, namely $M^\varepsilon_n(z)$, possessing the desired boundedness. Then, the uniform convergence of the non-negligible entries of $M^\varepsilon_n(z)$ on $\mathcal{D}_2$ can be split into two parts: (1)  a concentration estimate of $M_n(z)(i,j)$ around its mean value $\mathbb{E}M^\varepsilon_n(z)(i,j)$; (2) convergence of $\mathbb{E}M^\varepsilon_n(z)(i,j)$. It turns out that the uniform convergence of $\mathbb{E}M^\varepsilon_n(z)(i,j)$ towards those of $M(z)(i,j)$ on $\mathcal{D}_2$ can be significantly reduced to that on the interval $[1+\eta,2]$ (say). For the latter, we can prove it simply by using the Stieltjes transform and R-transform scheme. In the sequel, we unfold this strategy in detail.
\subsection{Modification of $M_n(z)$}
 For simplicity, we use the notation
\begin{eqnarray*}
\tilde{\Phi}(z)\equiv\tilde{\Phi}_n(z):=(1-z)\Phi(z).
\end{eqnarray*}
Now, according to  (\ref{071603})-(\ref{0604}), the entries in $M_n(z)-I$ are multiples of the quantities of the following three types,
\begin{eqnarray}
\mathbf{e}_i'\tilde{\Phi}(z)\mathbf{u}_j,\quad \mathbf{e}_i'\tilde{\Phi}(z)\mathbf{e}_j,\quad \mathbf{u}_i'\tilde{\Phi}(z)\mathbf{u}_j, \label{062902}
\end{eqnarray}
 and the multipliers can only be $\chi_{ij}$ or $t_i$, where $i,j\in\{1,\ldots, k\}$. A very elementary observation on $\chi_{ij}$ is as follows.
For any given $\varepsilon>0$, we have
 \begin{eqnarray}
|\chi_{ij}-\delta_{ij}t_i t_j|\leq \varepsilon \label{lem.071601}
 \end{eqnarray}
 with overwhelming probability, where $\delta_{ij}$ stands for the Kronecker delta function.
In the sequel, we will introduce some slight modifications on $\mathbf{e}_i'\tilde{\Phi}(z)\mathbf{e}_i$ and  $\mathbf{u}_i'\tilde{\Phi}(z)\mathbf{u}_i$. We stress here, we will always keep the multipliers $\chi_{ij}$'s and $t_i$'s unchanged in the modifications. Hence, we just lazily regard the terms listed in (\ref{062902}) as the entries of $M_n(z)-I$ in the sequel. 

For brevity, we simply write $\Phi(z)$ and $\tilde{\Phi}(z)$ as $\Phi$ and $\tilde{\Phi}$ when there is no confusion. By definitions, we have
\begin{eqnarray}
\mathbf{e}_i'\tilde{\Phi}\mathbf{u}_j=(\tilde{\Phi}S_{\mathbf{w}\mathbf{y}})(i,j), \quad \mathbf{e}_i'\tilde{\Phi}\mathbf{e}_j=\tilde{\Phi}(i,j),\quad \mathbf{u}_i'\tilde{\Phi}\mathbf{u}_j=(S_{\mathbf{y}\mathbf{w}}\tilde{\Phi} S_{\mathbf{w}\mathbf{y}})(i,j), \label{071001}
\end{eqnarray}
At first, we perform a deduction for the matrix $S_{\mathbf{y}\mathbf{w}}\Phi S_{\mathbf{w}\mathbf{y}}$. To this end, we introduce the matrix
\begin{eqnarray*}
\Upsilon\equiv\Upsilon(z):=(S_{\mathbf{y}\mathbf{w}}S_{\mathbf{w}\mathbf{w}}^{-1}S_{\mathbf{w}\mathbf{y}}-z S_{\mathbf{y}\mathbf{y}})^{-1}.
\end{eqnarray*} 
It follows from the definition of $\Phi$ that
\begin{eqnarray*}
&&S_{\mathbf{y}\mathbf{y}}^{-1/2}S_{\mathbf{y}\mathbf{w}}\Phi S_{\mathbf{w}\mathbf{y}}S_{\mathbf{y}\mathbf{y}}^{-1/2}=S_{\mathbf{y}\mathbf{y}}^{-1/2}S_{\mathbf{y}\mathbf{w}}S_{\mathbf{w}\mathbf{w}}^{-1/2}(S_{\mathbf{w}\mathbf{w}}^{-1/2}S_{\mathbf{w}\mathbf{y}}S_{\mathbf{y}\mathbf{y}}^{-1}S_{\mathbf{y}\mathbf{w}}S_{\mathbf{w}\mathbf{w}}^{-1/2}-z I_p)^{-1}S_{\mathbf{w}\mathbf{w}}^{-1/2}S_{\mathbf{w}\mathbf{y}}S_{\mathbf{y}\mathbf{y}}^{-1/2}\nonumber\\\\
&&=S_{\mathbf{y}\mathbf{y}}^{-1/2}S_{\mathbf{y}\mathbf{w}}S_{\mathbf{w}\mathbf{w}}^{-1}S_{\mathbf{w}\mathbf{y}}S_{\mathbf{y}\mathbf{y}}^{-1/2}(S_{\mathbf{y}\mathbf{y}}^{-1/2}S_{\mathbf{y}\mathbf{w}}S_{\mathbf{w}\mathbf{w}}^{-1}S_{\mathbf{w}\mathbf{y}}S_{\mathbf{y}\mathbf{y}}^{-1/2}-zI_q)^{-1}= I_q+zS_{\mathbf{y}\mathbf{y}}^{1/2}\Upsilon S_{\mathbf{y}\mathbf{y}}^{1/2},
\end{eqnarray*}
from which we immediately get the identity
\begin{eqnarray}
S_{\mathbf{y}\mathbf{w}}\Phi S_{\mathbf{w}\mathbf{y}}=S_{\mathbf{y}\mathbf{y}}+z S_{\mathbf{y}\mathbf{y}}\Upsilon S_{\mathbf{y}\mathbf{y}}. \label{070201}
\end{eqnarray}
Now we introduce the following two $n$ by $n$ projection matrices
\begin{eqnarray*}
P_\mathbf{w}:=W'(WW')^{-1}W,\quad P_\mathbf{y}:=Y'(YY')^{-1}Y.
\end{eqnarray*}
Obviously, $\text{rank}(P_\mathbf{w})=p$ and $\text{rank}(P_\mathbf{y})=q$ almost surely. Further, we define the following scaled Wishart matrices
\begin{eqnarray*}
E:=n^{-1} WP_\mathbf{y}W', \quad H:=n^{-1} W(I_n-P_\mathbf{y})W',\quad \mathcal{E}:=n^{-1}YP_\mathbf{w}Y',\quad \mathcal{H}:=n^{-1} Y(I_n-P_\mathbf{w})Y'.
\end{eqnarray*}
By Cochran's Theorem, we know that $E$ and $H$ are independent, so are $\mathcal{E}$ and $\mathcal{H}$. In addition,
\begin{eqnarray*}
&&nE\sim \text{Wishart}_p(I_p, q),\quad nH\sim\text{Wishart}_p(I_p, n-q),\nonumber\\
&& n\mathcal{E}\sim \text{Wishart}_q( I_q, p),\quad n\mathcal{H}\sim \text{Wishart}_q(I_q, n-p).
\end{eqnarray*}
By the above notation, we have the following representations:
\begin{eqnarray*}
S_{\mathbf{w}\mathbf{y}}S_{\mathbf{y}\mathbf{y}}^{-1}S_{\mathbf{y}\mathbf{w}}=E,\quad S_{\mathbf{w}\mathbf{w}}=E+H, \quad S_{\mathbf{y}\mathbf{w}}S_{\mathbf{w}\mathbf{w}}^{-1}S_{\mathbf{w}\mathbf{y}}=\mathcal{E},\quad S_{\mathbf{y}\mathbf{y}}=\mathcal{E}+\mathcal{H},
\end{eqnarray*}
which implies
\begin{eqnarray}
&&\Phi=(E-zS_{\mathbf{w}\mathbf{w}})^{-1}=((1-z)E-z H)^{-1}, \nonumber\\ \nonumber\\
 &&\Upsilon=(\mathcal{E}-zS_{\mathbf{y}\mathbf{y}})^{-1}=((1-z)\mathcal{E}-z \mathcal{H})^{-1},\label{061911}
\end{eqnarray}
and thus
\begin{eqnarray}
S_{\mathbf{y}\mathbf{y}}\Upsilon S_{\mathbf{y}\mathbf{y}}&=&(\mathcal{E}+\mathcal{H})((1-z)\mathcal{E}-z \mathcal{H})^{-1}(\mathcal{E}+\mathcal{H})\nonumber\\ \nonumber\\
&=&(1-z)^{-1}\mathcal{E}-z^{-1}\mathcal{H}+(z(1-z))^{-1}((1-z)\mathcal{H}^{-1}-z \mathcal{E}^{-1})^{-1} \label{061912}
\end{eqnarray}
Following from (\ref{070201}) and (\ref{061912}) we obtain
\begin{eqnarray*}
S_{\mathbf{y}\mathbf{w}}\tilde{\Phi}S_{\mathbf{w}\mathbf{y}}=(1-z)S_{\mathbf{y}\mathbf{w}}\Phi S_{\mathbf{w}\mathbf{y}}=\mathcal{E}+((1-z)\mathcal{H}^{-1}-z \mathcal{E}^{-1})^{-1}.
\end{eqnarray*}
In the sequel, we use the notation
\begin{eqnarray*}
\Psi\equiv \Psi_n(z):=((1-z)\mathcal{H}^{-1}-z \mathcal{E}^{-1})^{-1},
\end{eqnarray*}
thus $S_{\mathbf{y}\mathbf{w}}\tilde{\Phi}S_{\mathbf{w}\mathbf{y}}=\mathcal{E}+\Psi$.  In light of (\ref{071001}), it suffices to study the following quantities
\begin{eqnarray*}
(\tilde{\Phi}S_{\mathbf{w}\mathbf{y}})(i,j), \quad \tilde{\Phi}(i,j),\quad \mathcal{E}(i,j)+\Psi(i,j).
\end{eqnarray*}
Now, we introduce  the modifications of $\Phi $ and $\Psi$ we will work with in the sequel. Define for any fixed small positive $\varepsilon>0$
\begin{eqnarray*}
\Phi^\varepsilon :=((1-z)E-z H-z\varepsilon I_p)^{-1},\quad \Psi^\varepsilon:= ((1-z)\mathcal{H}^{-1}-z \mathcal{E}^{-1}-z\varepsilon I_q))^{-1}
\end{eqnarray*}
Analogously, we write $\tilde{\Phi}^{\varepsilon}=(1-z)\Phi^\varepsilon$. Accordingly, we define a modification of $M_n(z)$, namely, $M^\varepsilon_n(z)$, obtained by the following replacement on the main factors of the non-negligible entries of $M_n(z)$:
\begin{eqnarray*}
\tilde{\Phi} (i,i)\to \tilde{\Phi}^{\varepsilon}(i,i),\quad (S_{\mathbf{y}\mathbf{w}}\Phi S_{\mathbf{w}\mathbf{y}})(i,i)\to \mathcal{E}(i,i)+\Psi^\varepsilon(i,i), \quad \text{for all} \quad i=1,\ldots, k.
\end{eqnarray*}
As mentioned above, we do not alter the multipliers $\chi_{ij}$'s and $t_i$'s in the entries. Moreover, we also retain $ (\tilde{\Phi} S_{\mathbf{w}\mathbf{y}})(i,j)$ for all $(i,j)$ pair, as well as $\tilde{\Phi} (i,j)$ and $(S_{\mathbf{y}\mathbf{w}}\Phi S_{\mathbf{w}\mathbf{y}})(i,j)$ with $i\neq j$ in the modification.
For technical reason, we also need another variants of $\Phi$ and $\Psi$, namely $\mathring{\Phi}\equiv \mathring{\Phi}_n(z)$ and $\mathring{\Psi}\equiv\mathring{\Psi}_n(z)$. To define them, we write the spectral decomposition of $E$, $H$, $\mathcal{E}$ and $\mathcal{H}$ by
\begin{eqnarray*}
E=U_{e1}\Lambda_{e1}U_{e1}', \quad H=U_{h1}\Lambda_{h1}U_{h1}',\quad \mathcal{E}=U_{e2}\Lambda_{e2}U_{e2}', \quad H=U_{h2}\Lambda_{h2}U_{h2}'.
\end{eqnarray*}
Now we introduce the truncation function $\chi_a^b:\mathbb{R}\to\mathbb{R}$ with two real numbers $a<b$, defined as 
\begin{eqnarray*}
\chi_a^b(x):=\left\{
\begin{array}{ccc}
x, &\text{if}\quad x\in[a,b],\\
a, &\text{if}\quad x< a,\\
b, &\text{if}\quad x>b
\end{array}
\right.
\end{eqnarray*}
Now we introduce three positive numbers
 \begin{eqnarray*}
 \kappa_1:=\frac{(\sqrt{1-c_2}-\sqrt{c_1})^2}{2},\quad \kappa_2:= \frac{(\sqrt{c_1}-\sqrt{c_2})^2}{2},\quad \kappa_3:=\frac{(\sqrt{1-c_1}-\sqrt{c_2})^2}{2},
 \end{eqnarray*}
 whereby we can define
\begin{eqnarray*}
&&\mathring{E}=\chi_0^4(E)=U_{e1}\chi_0^4(\Lambda_{e1})U_{e1}',\quad \mathring{H}=\chi_{\kappa_1}^4(H)=U_{h1}\chi_{\kappa_1}^4(\Lambda_{h1})U_{h1}'\nonumber\\
&&\mathring{\mathcal{E}}=\chi_{\kappa_2}^4(\mathcal{E})=U_{e2}\chi_{\kappa_2}^4(\Lambda_{e2})U_{e2}', \quad \mathring{\mathcal{H}}=\chi_{\kappa_3}^4(\mathcal{H})=U_{h2}\chi_{\kappa_3}^4(\Lambda_{h2})U_{h2}'.
\end{eqnarray*}
It  is easy to check that $\mathring{E}=E$, $\mathring{H}=H$, $\mathring{\mathcal{E}}=\mathcal{E}$ and $\mathring{\mathcal{H}}=\mathcal{H}$ all hold with overwhelming probability, according to (\ref{0726100}) and Assumption \ref{ass.070301}. Under these notations, we can define
\begin{eqnarray*}
\mathring{\Phi}:=((1-z)\mathring{E}-z\mathring{H})^{-1},\quad \mathring{\Psi}:=((1-z)\mathring{\mathcal{H}}^{-1}-z\mathring{\mathcal{E}}^{-1})^{-1}.
\end{eqnarray*} 
Consequently, one has $\mathring{\Phi}=\Phi$ and $\mathring{\Psi}=\Psi$ both hold with overwhelming probability. Correspondingly, we set $\mathring{\tilde{\Phi}}:=(1-z)\mathring{\Phi}$.

Now, setting
\begin{eqnarray*}
\mathcal{D}_3\equiv\mathcal{D}_3(\eta,\varepsilon):=\{z\in \mathbb{C}: d_r+\frac{3}{2}\eta< \Re z< 1+\eta, |\Im z|<\varepsilon\}\subset\mathcal{D}_2,
\end{eqnarray*}
we have the following lemma, which provides us necessary bounds on $||\Phi||$,  $||\Psi||$, $||\Phi^\varepsilon||$, $||\Psi^\varepsilon||$, $||\mathring{\Phi}||$ and $||\mathring{\Psi}||$, coming in handy below, and controls the difference 
between $\Phi$ and $\Phi^\varepsilon$ as well as that between $\Psi$ and  $\Psi^\varepsilon$. The latter then guarantees the validity of our modification on $M_n(z)$.


\begin{lem} \label{lem.071001}For any given $\eta>0$ in the definition of $\mathcal{D}_2$ and $\mathcal{D}_3$, and any $\varepsilon>0$, there exist some positive constants $C_1$, $C_2$, $C_3$ and $C_4$(depending on $\eta$ only) such that the following four statements hold.
\begin{enumerate}
\item[(i):] 
\hspace{20ex}$
\displaystyle \sup_{z\in\mathcal{D}_2}\max\{||\Phi(z)||, ||\Phi^\varepsilon(z)||, ||\Psi(z)||, ||\Psi^\varepsilon(z)||\}\leq C_1
$\\
holds with overwhelming probability.\\
\item[(ii):] 
\hspace{20ex}$ \displaystyle \sup_{z\in\mathcal{D}_2}\max\{||\Phi(z)-\Phi^\varepsilon(z)||, ||\Psi(z)-\Psi^\varepsilon(z)||\}\leq C_2\varepsilon
$\\
holds with overwhelming probability.\\
\item[(iii):] \hspace{20ex}$\displaystyle \sup_{z\in\mathcal{D}_2\setminus\mathcal{D}_3}\max\{||\Phi^\varepsilon(z)||, ||\Psi^\varepsilon(z)||\}\leq C_3(\varepsilon^{-1}+\varepsilon^{-2})
$\\
holds deterministically.\\
\item[(iv):] \hspace{20ex}$\displaystyle \sup_{z\in\mathcal{D}_2\setminus\mathcal{D}_3}\max\{||\mathring{\Phi}(z)||, ||\mathring{\Psi}(z)||\}\leq C_4 \varepsilon^{-1}
$\\
holds deterministically.
\end{enumerate}
\end{lem}
The proof of Lemma \ref{lem.071001} is stated in the Appendix. 
 Now, with Lemma \ref{lem.071001} granted, we have the following lemma.
\begin{lem} \label{lem.062901}For any given $\eta>0$ in the definition of $\mathcal{D}_2$, and any $\varepsilon>0$, there exists some sufficiently large positive constant $C$ (depending on $\eta$ only), such that 
\begin{eqnarray}
\sup_{z\in \mathcal{D}_2}\max_{i,j=1,\ldots,k^2+2k}|M_n(z)(i,j)-M_n^\varepsilon(z)(i,j)|\leq C\varepsilon 
\end{eqnarray} 
holds with overwhelming probability.
\end{lem}
\begin{proof}
According to the above discussion and construction, it is easy to see that
\begin{eqnarray*}
|M_n(z)(i,j)-M_n^\varepsilon(z)(i,j)|\leq \max_{i,j}|\chi_{ij}|\cdot \max_{i}t_i \cdot \max\{||\Phi(z)-\Phi^\varepsilon(z)||, ||\Psi(z)-\Psi^\varepsilon(z)||\}
\end{eqnarray*}
Note that when $r_i<1$, $t_i$ is bounded. Therefore, using (\ref{lem.071601}) and  (ii) of Lemma \ref{lem.071001} yields the conclusion.
\end{proof}
Then, what remains is to show the following lemma.
\begin{lem} \label{lem.062904}For any given $\eta>0$ in the definition of $\mathcal{D}_2$, and any $\varepsilon>0$, there exists some sufficiently large positive constant $C$ (depending on $\eta$ only), such that
\begin{eqnarray}
\sup_{z\in \mathcal{D}_2}\sup_{i,j=1,\ldots,k^2+2k}|M_n^\varepsilon(z)(i,j)-M(z)(i,j)|\leq C\varepsilon \label{070202}
\end{eqnarray}
holds with overwhelming probability.
\end{lem} 
For Lemma \ref{lem.062904}, we can further decompose it into two lemmas, on the negligible entries and the non-negligible entries of $M_n^\varepsilon(z)$ respectively. The first one is as follows.
\begin{lem} \label{lem.071003}
For any given $\eta>0$ and $\varepsilon>0$, there exists some sufficiently large positive constant $C$ (depending on $\eta$ only), such that for any $i,j\in \{1,\ldots, k\}$, we have
\begin{eqnarray}
\sup_{z\in \mathcal{D}_2}|(\tilde{\Phi} S_{\mathbf{w}\mathbf{y}})(i,j)|\leq C\varepsilon \label{072506}
\end{eqnarray}
with overwhelming probability.
If $i\neq j$, there exists some positive constant $C$ (depending on $\eta$ only), such that
\begin{eqnarray}
\sup_{z\in \mathcal{D}_2}\max \{ |\Phi (i,j)|, |\mathcal{E}(i,j)|, |\Psi(i,j)|\}\leq C\varepsilon \label{072507}
\end{eqnarray}
holds with overwhelming probability.
\end{lem}
Now, what remains is to estimate $\tilde{\Phi}^\varepsilon(i,i)$, $\Psi^\varepsilon(i,i)$ and $\mathcal{E}(i,i)$. Akin to (\ref{lem.071601}), it is elementary to see that
\begin{eqnarray}
|\mathcal{E}(i,j)-c_1\delta_{ij}|\leq \varepsilon \label{071901}
\end{eqnarray}
with overwhelming probability.
Setting 
\begin{eqnarray*}
\varrho(z)=h(z)-c_1,
\end{eqnarray*}
we have the following lemma  on the non-negligible entries.
\begin{lem}  \label{lem.071004}
For any given $\eta>0$ and $\varepsilon>0$, there exists some sufficiently large positive constant $C$ (depending on $\eta$ only), such that for each $i\in \{1,\ldots, k\}$, we have
\begin{eqnarray*}
\sup_{z\in\mathcal{D}_2}\max\{|\tilde{\Phi}^\varepsilon(i,i)-f(z)|, |\Psi^\varepsilon(i,i)-\varrho(z)|\}\leq C\varepsilon
\end{eqnarray*}
holds with overwhelming probability.
\end{lem}
\subsection{Reduction for Lemma \ref{lem.071004}}
To facilitate the proof, we further reduce Lemma \ref{lem.071004} to two more easily handled tasks, namely Lemma \ref{lem.071007} and \ref{lem.071008}.  Roughly, the former is on the concentration estimate of  $\Phi^\varepsilon(i,i)$ and $\Psi^\varepsilon(i,i)$ around their expectations, and the latter provides the limits of the expectations. However, in both two steps, we reduce the region $\mathcal{D}_2$ to more restricted ones, which are more friendly to our reasoning, meanwhile, it will be clear that such slimmed-down discussions are sufficient to imply Lemma \ref{lem.071004}.
\begin{lem}  \label{lem.071007}For any given $\eta>0$ and $\varepsilon>0$, there exists some sufficiently large positive constant $C$ (depending on $\eta$ only), such that for all $i\in \{1,\ldots, k\}$, we have 
\begin{eqnarray*}
\sup_{z\in\mathcal{D}_2\setminus\mathcal{D}_3}\max\{|\Phi^\varepsilon(i,i)-\mathbb{E}\Phi^\varepsilon(i,i)|,  |\Psi^\varepsilon(i,i)-\mathbb{E}\Psi^\varepsilon(i,i)|\}\leq C\varepsilon
\end{eqnarray*}
with overwhelming probability. 
\end{lem}
\begin{lem} \label{lem.071008} Given any $\eta>0$ and $\varepsilon>0$, we have
\begin{eqnarray}
\sup_{z\in\mathcal{D}_2\setminus\mathcal{D}_3}|\mathbb{E}\tilde{\Phi}^\varepsilon(i,i)-\mathbb{E}\mathring{\tilde{\Phi}}(i,i)|\leq C_1\varepsilon,\quad \sup_{z\in\mathcal{D}_2\setminus\mathcal{D}_3}|\mathbb{E}\Psi^\varepsilon(i,i)-\mathbb{E}\mathring{\Psi}(i,i)|\leq C_2\varepsilon \label{071010}
\end{eqnarray}
with some positive constants $C_1$ and $C_2$ (depending on $\eta$ only). Moreover, for any $z\in [1+\eta,2]$,
\begin{eqnarray}
\mathbb{E}\mathring{\tilde{\Phi}}(i,i)\to f(z),\quad \mathbb{E}\mathring{\Psi}(i,i)\to \varrho(z). \label{071011}
\end{eqnarray}
\end{lem}
Now, we prove Lemma \ref{lem.071004} with Lemmas \ref{lem.071007} and \ref{lem.071008} granted.

\begin{proof}[Proof of Lemma \ref{lem.071004}]
Note that, obviously $\mathring{\tilde{\Phi}}(i,i)$ and $\mathring{\Psi}(i,i)$ are both holomorphic on $\mathcal{D}_2\setminus\mathcal{D}_3$, by definitions, so are $\mathbb{E}\mathring{\tilde{\Phi}}(i,i)$ and $\mathbb{E}\mathring{\Psi}(i,i)$. In addition, by (iv) of Lemma \ref{lem.071001}, it follows from Montel's theorem that both $\{\mathbb{E}\mathring{\tilde{\Phi}}_n(i,i)\}_{n\geq1}$ and $\{\mathbb{E}\mathring{\Psi}_n(i,i)\}_{n\geq1}$ are normal families on ${\mathcal{D}}_2\setminus \mathcal{D}_3$. Hence, each subsequence of $\{\mathbb{E}\mathring{\tilde{\Phi}}_n(i,i)\}_{n\geq1}$ (resp. $\{\mathbb{E}\mathring{\Psi}_n(i,i)\}_{n\geq1}$) contains a further subsequence converging uniformly on each compact subset of ${\mathcal{D}}_2\setminus \mathcal{D}_3$ to a holomorphic function. Selecting one of such convergent subsequences, we call its limit $\hat{f}(z)$ (resp. $\hat{\varrho}(z)$). Now, according to (\ref{071011}), we already know that $\mathbb{E}\mathring{\tilde{\Phi}}_n(i,i)\to f(z)$ and $\mathbb{E}\mathring{\Psi}_n(i,i)\to \varrho(z)$ for all $z\in[1+\eta,2]$, which imply that $\hat{f}(z)=f(z)$ and $\hat{\varrho}(z)=\varrho(z)$ on $[1+\eta,2]$. Then Identity Theorem leads us to the fact that $\hat{f}(z)=f(z)$ and $\hat{\varrho}(z)=\varrho(z)$ on each compact subset of $\mathcal{D}_2\setminus\mathcal{D}_3$. Hence we get the fact that $\mathbb{E}\mathring{\tilde{\Phi}}(i,i)\to f(z)$ and $\mathbb{E}\mathring{\Psi}(i,i)\to \varrho (z)$ uniformly on each compact subset of $\mathcal{D}_2\setminus\mathcal{D}_3$. By slightly adjusting the value of $\eta$ and $\varepsilon$, we can safely say that $\mathbb{E}\mathring{\tilde{\Phi}}(i,i)\to f(z)$ and $\mathbb{E}\mathring{\Psi}(i,i)\to \varrho (z)$ uniformly on $\mathcal{D}_2\setminus\mathcal{D}_3$. Combining this fact with (\ref{071010}) and Lemma \ref{lem.071007} yields  that for some positive constant $C$
\begin{eqnarray*}
\sup_{z\in\mathcal{D}_2\setminus\mathcal{D}_3}\max\{|\tilde{\Phi}^\varepsilon(i,i)-f(z)|, |\Psi^\varepsilon(i,i)-\varrho(z)|\}\leq C\varepsilon
\end{eqnarray*}
holds with overwhelming probability.  What remains is to extend the estimate to the whole $\mathcal{D}_2$. Now for any $z_0\in \mathcal{D}_3$, we define $z_0^\varepsilon=\Re z_0+\sqrt{-1}\varepsilon$. Then it suffices to show
\begin{eqnarray}
\sup_{z_0\in \mathcal{D}_3}\max\{||(\tilde{\Phi}^\varepsilon_n(z_0))-(\tilde{\Phi}^\varepsilon_n(z_0^\varepsilon))||, ||(\Psi^\varepsilon_n(z_0))-(\Psi^\varepsilon_n(z_0^\varepsilon))||\}\leq C\varepsilon \label{071014}
\end{eqnarray}
with overwhelming probability, and 
\begin{eqnarray}
\sup_{z_0\in \mathcal{D}_3}\max\{|(f(z_0)-f(z_0^\varepsilon)|, |\varrho(z_0)-\varrho(z_0^\varepsilon)|\}\leq C\varepsilon. \label{071015}
\end{eqnarray}
Note that (\ref{071015}) is easy to check. Thus we only verify (\ref{071014}) below. We will state the proof for $\tilde{\Phi}^\varepsilon$ only, the case of $\Psi^\varepsilon$ is similar.  By definition, we have
\begin{eqnarray}
&&||\tilde{\Phi}^\varepsilon(z_0)-\tilde{\Phi}^\varepsilon(z_0^\varepsilon)||\nonumber\\
&&\hspace{5ex}= |(\varepsilon-\Im z_0)|\cdot (||\Phi^\varepsilon(z_0)||+||(1-z_0^\varepsilon) \Phi^\varepsilon(z_0^\varepsilon)(E+H+\varepsilon I_p)\Phi^\varepsilon(z_0))||)\leq C\varepsilon \label{071143}
\end{eqnarray}
with overwhelming probability, by taking into account (i) of Lemma \ref{lem.071001} and the fact that $E+H=S_{\mathbf{w}\mathbf{w}}$ is bounded in operator norm with overwhelming probability, according to (\ref{0726100}). Hence, we conclude the proof.
\end{proof}
Consequently, it suffices to prove Lemmas \ref{lem.071003}, \ref{lem.071007} and \ref{lem.071008} in the remaining part of this paper. We will handle Lemma \ref{lem.071003} on the negligible entries in Section 6 and Lemmas \ref{lem.071007} and \ref{lem.071008} on the non-negligible entries in Section 7.\\\\
\section{Proof of Lemma \ref{lem.071003}} In this section, we will frequently need the following rough large deviation inequalities on the columns of a Haar distributed orthogonal matrix.
\begin{lem} \label{lem.071604} Assume that $\boldsymbol{\upsilon}_i$ and $\boldsymbol{\upsilon}_j$ ($i\neq j$) are respectively $i$th and $j$th columns  of an $n$ by $n$ random orthogonal matrix $\mathcal{U}_n$, which is Haar distributed on the orthogonal group $\mathcal{O}(n)$. Let $A$ be an $n$ by $n$ matrix and $\mathbf{b}$ be an $n$-dimensional complex vector, both independent of $\mathcal{U}_n$. Moreover, $||A||, ||\mathbf{b}||\leq C$ for some positive constant $C$. The following three statements hold with overwhelming probability for any given $\varepsilon>0$:
\begin{eqnarray*}
(i): |\boldsymbol{\upsilon}_i'A\boldsymbol{\upsilon}_i-\frac{1}{n}\text{tr} A|\leq \varepsilon , \quad (ii): |\boldsymbol{\upsilon}_i'A\boldsymbol{\upsilon}_j|\leq \varepsilon,\quad (iii): |\boldsymbol{\upsilon}_i'\mathbf{b}|\leq \varepsilon
\end{eqnarray*}
\end{lem}
We state the proof of Lemma \ref{lem.071604} in the Appendix, which relies on an approximation of columns of Haar distributed matrix by independent Gaussian vectors, see \cite{Jiang2005}, for instance.
\subsection{Bound on $\Phi S_{\mathbf{wy}}(i,j)$}
At first, we recall the definition
\begin{eqnarray*}
\Phi S_{\mathbf{wy}}=((1-z)S_{\mathbf{wy}}S_{\mathbf{yy}}^{-1}S_{\mathbf{yw}}-z S_{\mathbf{ww}})^{-1}S_{\mathbf{wy}}
\end{eqnarray*}
Writing the SVDs of $W/\sqrt{n}$ and $Y/\sqrt{n}$ as
\begin{eqnarray*}
W/\sqrt{n}=U_\mathbf{w}\Lambda_{\mathbf{w}} V_{\mathbf{w}}',\quad Y/\sqrt{n}=U_\mathbf{y}\Lambda_{\mathbf{y}} V_{\mathbf{y}}',
\end{eqnarray*}
and introducing the additional notations
\begin{eqnarray*}
D_\mathbf{w}:=\Lambda_{\mathbf{w}}\Lambda_{\mathbf{w}}', \quad \hat{Y}:=YV_{\mathbf{w}},\quad P_{\hat{\mathbf{y}}}:=\hat{Y}'(\hat{Y}\hat{Y}')^{-1}\hat{Y},
\end{eqnarray*}
one has $\hat{Y}\stackrel{d}=Y$ thus $P_{\hat{\mathbf{y}}}\stackrel{d}=P_{\mathbf{y}}$, and
\begin{eqnarray*}
\Phi S_{\mathbf{wy}}=U_\mathbf{w}((1-z)\Lambda_{\mathbf{w}}P_{\hat{\mathbf{y}}}\Lambda_{\mathbf{w}}'-z D_\mathbf{w})^{-1}\Lambda_{\mathbf{w}}\hat{Y}'/\sqrt{n}.
\end{eqnarray*}
Now note that $W$ is left orthogonally invariant (in distribution), consequently,  if $\tilde{U}$ is a Haar distributed orthogonal matrix independent of $W$ and $Y$, $\tilde{U}W$ and $W$ are identically distributed and $\tilde{U}W$ is independent of $Y$ obviously. In addition, note that the conditional distribution of $\tilde{U}U_\mathbf{w}$ given $\Lambda_\mathbf{w}$, $V_\mathbf{w}$ and $Y$ does not depend on the realization of $\Lambda_\mathbf{w}$, $V_\mathbf{w}$ and $Y$, owing to the fact that $\tilde{U}$ is Haar distributed. That means, $\tilde{U}U_\mathbf{w}$ is independent of $\Lambda_\mathbf{w}$, $V_\mathbf{w}$ and $Y$. Consequently, since $\tilde{U}W$ and $W$ are identically distributed, we arrive at the conclusion that $U_\mathbf{w}$ is independent of $\Lambda_\mathbf{w}$, $V_\mathbf{w}$ and $Y$.
Now, setting 
\begin{eqnarray*}
\Pi(z):=((1-z)\Lambda_{\mathbf{w}}P_{\hat{\mathbf{y}}}\Lambda_{\mathbf{w}}'-z D_\mathbf{w})^{-1}\Lambda_{\mathbf{w}}\hat{Y}/\sqrt{n},
\end{eqnarray*}
we  have $
\Phi S_{\mathbf{wy}}=U_\mathbf{w}\Pi(z),
$
where $U_\mathbf{w}$ is independent of $\Pi(z)$, by the above discussion. 
 Let $\mathbf{u}_{\mathbf{w},i}$ be the transpose of the $i$th row of $U_\mathbf{w}$, and $\boldsymbol{\pi}_j(z)$ be the $j$th column of $\Pi(z)$.  
Then
\begin{eqnarray*}
(\Phi S_{\mathbf{wy}})(i,j)=\mathbf{u}_{\mathbf{w},i}'\boldsymbol{\pi}_j(z)
\end{eqnarray*} 
Now, note that 
\begin{eqnarray}
||((1-z)\Lambda_{\mathbf{w}}P_{\hat{\mathbf{y}}}\Lambda_{\mathbf{w}}'-z D_\mathbf{w})^{-1}||=||\Phi(z)||, \label{072503}
\end{eqnarray} 
hence is bounded in $\mathcal{D}_2$ with overwhelming probability, according to (i) of Lemma \ref{lem.071001}. It follows from (\ref{0726100}) that $||\Lambda_{\mathbf{w}}||$ and $||\hat{Y}/\sqrt{n}||$ are bounded with overwhelming probability. Consequently, $||\Pi(z)||$ is bounded uniformly on $\mathcal{D}_2$ with overwhelming probability, which implies that $||\pi_j(z)||$ is bounded
with overwhelming probability. Applying (iii) of Lemma \ref{lem.071604}  yields that 
\begin{eqnarray}
|(\Phi(z) S_{\mathbf{wy}})(i,j)|\leq \varepsilon  \label{072504}
\end{eqnarray} with overwhelming probability for any $z\in\mathcal{D}_2$.  To strengthen the bound to be uniform on $\mathcal{D}_2$, we introduce the event 
\begin{eqnarray*}
\Xi_4\equiv\Xi_4(n,K, C)&:=&\bigg{\{}\max\{||E||,||H||, ||\mathcal{E}^{-1}||,||\mathcal{H}^{-1}||, ||W/\sqrt{n}||, ||Y/\sqrt{n}||\}\leq K\bigg{\}}\nonumber\\
&&\hspace{-2ex}\bigcap\bigg{\{}\sup_{z\in\mathcal{D}_2}\max\{||\Phi(z)||, ||\Phi^\varepsilon(z)||, ||\Psi(z)||, ||\Psi^\varepsilon(z)||\}\leq C\bigg{\}}.
\end{eqnarray*}
Note that for sufficiently large $K$ and $C$, $\Xi_4$ holds with overwhelming probability by Assumption \ref{ass.070301} and (i) of Lemma \ref{lem.071001}. Now, we assign an lattice on $\mathcal{D}_2$ of size $O(\varepsilon^{-2})$ such that each two adjacent points in this lattice have a distance $\varepsilon$ (say). 
In addition, it is easy too see that for two adjacent points $z_1$ and $z_2$ in the lattice, 
\begin{eqnarray}
&&||(\Phi(z_1)-\Phi(z_2))S_{\mathbf{wy}}||\nonumber\\ &&\hspace{7ex}\leq n^{-1}||W||\cdot||Y||\cdot ||\Phi(z_1)||\cdot ||\Phi(z_2)||\cdot \varepsilon (||E||+||H||+\varepsilon)\leq C'\varepsilon  \label{072505}
\end{eqnarray}
in $\Xi_4$ with some positive constant $C'$. In addition, by the definition of {\emph{overwhelming probability}}, one see that (\ref{072504}) holds uniformly on the lattice with overwhelming probability. Then in $\Xi_4$, this bound can be further extended to the whole $\mathcal{D}_2$ by using  (\ref{072505}). Hence, we finally get the conclusion that (\ref{072504}) holds uniformly on $\mathcal{D}_2$ with overwhelming probability, i.e. (\ref{072506}) is valid.
\subsection{Bound on $|\Phi (i,j)|$, $|\mathcal{E}(i,j)|$, $|\Psi(i,j)|$ when $i\neq j$}
At first, according to (\ref{071901}), the bound on $|\mathcal{E}(i,j)|$ is valid. It suffices to check that for $|\Phi (i,j)|$ and $|\Psi(i,j)|$. Due to similarity, we just state the proof for $|\Phi (i,j)|$ in the sequel. Note that, by the notation introduced above, we see that
\begin{eqnarray*}
\Phi(z)=U_\mathbf{w}((1-z)\Lambda_{\mathbf{w}}P_{\hat{\mathbf{y}}}\Lambda_{\mathbf{w}}'-z D_\mathbf{w})^{-1}U_\mathbf{w}.
\end{eqnarray*}
Hence
\begin{eqnarray*}
\Phi(z)(i,j)=\mathbf{u}_{\mathbf{w},i}'((1-z)\Lambda_{\mathbf{w}}P_{\hat{\mathbf{y}}}\Lambda_{\mathbf{w}}'-z D_\mathbf{w})^{-1} \mathbf{u}_{\mathbf{w},j}.
\end{eqnarray*}
Then, using (\ref{072503}), (i) of Lemma \ref{lem.071001} and  (ii) of Lemma \ref{lem.071604}, we obtain 
\begin{eqnarray}
|\Phi (i,j)|\leq C\varepsilon
\end{eqnarray} 
with overwhelming probability, for each $z\in\mathcal{D}_2$. In a similar vein, the uniform bound on $\mathcal{D}_2$ can be derived via a lattice argument as that for (\ref{072504}). We do not reproduce the details here.   Moreover,
the discussion on $|\Psi(i,j)|$ is just analogous.
\section{Proof of Lemmas \ref{lem.071007} and \ref{lem.071008}}
\subsection{Proof of Lemma \ref{lem.071007}}
To address the diagonal entries of $\Phi^\varepsilon$ and $\Psi^\varepsilon$, we need to use (i) of Lemma \ref{lem.071604}. At first, since $\Phi^\varepsilon$ and $\Psi^\varepsilon$ are orthogonally invariant, their distributions do not alter under conjugation by permutation matrices. Consequently, the diagonal entries of $\Phi^\varepsilon$ are identically distributed, so are those of $\Psi^\varepsilon$. That means
\begin{eqnarray}
\mathbb{E}\Phi^\varepsilon(i,i)=\frac1p\mathbb{E}\text{tr} \Phi^\varepsilon,\quad \mathbb{E}\Psi^\varepsilon(j,j)=\frac1q\mathbb{E}\text{tr} \Psi^\varepsilon\quad \text{for}\quad i=1,\ldots, p, \quad j=1,\ldots,q.\label{070305}
\end{eqnarray}
By the notation introduced in the last section, we can write
\begin{eqnarray*}
\Phi^\varepsilon(i,i)=\mathbf{u}_{\mathbf{w},i}' ((1-z)\Lambda_{\mathbf{w}}P_{\hat{\mathbf{y}}}\Lambda_{\mathbf{w}}'-z D_\mathbf{w}-z\varepsilon I_p)^{-1} \mathbf{u}_{\mathbf{w},i}
\end{eqnarray*}
Now, using (\ref{072503}), (i) of Lemma \ref{lem.071001} and (i) of Lemma \ref{lem.071604}, we see that for some positive constant $C$ and $z\in \mathcal{D}_2$, 
\begin{eqnarray}
|\Phi^\varepsilon(z)(i,i)-p^{-1}\text{tr} \Phi^\varepsilon(z)|\leq C\varepsilon \label{071701}
\end{eqnarray}
holds  overwhelming probability.  Analogously, we also have
\begin{eqnarray}
|\Psi^\varepsilon(i,i)(z)-q^{-1}\text{tr} \Psi^\varepsilon(z)|\leq C\varepsilon \label{071702}
\end{eqnarray}
with overwhelming probability.  Then the lattice argument used for (\ref{072504}) is also applicable, mutatis mutandis, to strengthen (\ref{071701}) and (\ref{071702})  to be uniform ones on $\mathcal{D}_2\setminus\mathcal{D}_3$. We omit the details here.  Consequently, in light of (\ref{070305}), it then 
suffices to show the following lemma.
\begin{lem} \label{lem.071101}We have the following two concentration inequalities, for any given $\varepsilon>0$,
\begin{eqnarray}
\sup_{z\in\mathcal{D}_2\setminus\mathcal{D}_3}|p^{-1}\text{tr} \Phi^\varepsilon(z)-p^{-1}\mathbb{E} \text{tr}\Phi^\varepsilon(z) |\leq \varepsilon, \label{070302}
\end{eqnarray}
and 
\begin{eqnarray}
\sup_{z\in\mathcal{D}_2\setminus\mathcal{D}_3}|q^{-1}\text{tr} \Psi^\varepsilon (z)-q^{-1}\mathbb{E} \text{tr}\Psi^\varepsilon (z) |\leq \varepsilon \label{070302}
\end{eqnarray}
hold with overwhelming probability.
\end{lem}
The proof of Lemma \ref{lem.071101}, based on a commonly used martingale difference technique,  will be postponed to the Appendix. 
Now, combining (\ref{070305})-(\ref{071702})  and Lemma \ref{lem.071101} yields Lemma \ref{lem.071007}.

\subsection{Proof of Lemma \ref{lem.071008}}
At first, we show (\ref{071010}). As mentioned, $\mathring{\Phi}=\Phi$ and $\mathring{\Psi}=\Psi$ for any $z$ with overwhelming probability.  It then follows from (ii) of Lemma \ref{lem.071001} that
\begin{eqnarray*}
\sup_{z\in\mathcal{D}_2}\max\{||\mathring{\Phi}(z)-\Phi^\varepsilon(z)||, ||\mathring{\Psi}(z)-\Psi^\varepsilon(z)||\}\leq C_2\varepsilon
\end{eqnarray*}
with overwhelming probability. Using the deterministic bounds in (iii) and (iv) of Lemma \ref{lem.071001} and the definition of {\emph{overwhelming probability}}, we can obtain (\ref{071010}).

Now, we we turn to (\ref{071011}). For simplicity, we introduce the notations
\begin{eqnarray}
\mathring{J}\equiv \mathring{J}_n(z):=(1-z)\mathring{E}-z\mathring{H}, \quad \mathring{\mathcal{J}}\equiv \mathring{\mathcal{J}}_n(z):= (1-z)\mathring{\mathcal{H}}^{-1}-z\mathring{\mathcal{E}}^{-1}. \label{071130}
\end{eqnarray}
Note that for $z\in[1+\eta,2]$, both $\mathring{J}$ and $\mathring{\mathcal{J}}$ are real symmetric and negative-definite, with probability 1. Then, we denote the Stieltjes transforms of the ESDs of $\mathring{J}$ and $\mathring{\mathcal{J}}$ at point $\omega$ by $s_{n1}(\omega)$ and $s_{n2}(\omega)$ respectively, it follows from (\ref{070305}) that
\begin{eqnarray*}
\mathbb{E}\mathring{\Phi}(i,i)=\mathbb{E}s_{n1}(0),\quad \mathbb{E}\mathring{\Psi}(i,i)=\mathbb{E}s_{n2}(0).
\end{eqnarray*}
By construction, we see that $\mathring{E}$, $\mathring{H}$, $\mathring{\mathcal{E}}^{-1}$ and $\mathring{\mathcal{H}}^{-1}$ are all bounded (in operator norm). Now, for given $z\in[1+\eta,2]$, let $s_1(\omega)$ and $s_2(\omega)$ be the Stieltjes transforms of the LSDs of $\mathring{J}$ and $\mathring{\mathcal{J}}$ respectively. Noticing that $\mathring{E}$, $\mathring{H}$, $\mathring{\mathcal{E}}^{-1}$ and $\mathring{\mathcal{H}}^{-1}$ are all orthogonally invariant, in light of \cite{Voiculescu1991}, we have
\begin{eqnarray*}
\mathbb{E}\Phi^\varepsilon(i,i)\to s_1(0),\quad \mathbb{E}\Psi^\varepsilon(i,i)\to s_2(0)
\end{eqnarray*}
when $n\to \infty$.  It remains to calculate $s_1(0)$ and $s_2(0)$ then. Now, we arrive at the stage to use (\ref{061802})-(\ref{061801}). To this end, we need the Stieltjes transforms for the LSDs of $(1-z)\mathring{E}$, $-z \mathring{H}$, $-z\mathring{\mathcal{E}}^{-1}$ and $(1-z)\mathring{\mathcal{H}}^{-1}$, denoted by $s_{e1}(\omega)$, $s_{h1}(\omega)$, $s_{e2}(\omega)$ and $s_{h2}(\omega)$ in the sequel, and the corresponding R-transforms will be represented by $R_{e1}(\omega)$, $R_{h1}(\omega)$, $R_{e2}(\omega)$ and $R_{h2}(\omega)$. Note that all these four Stieltjes transforms can be deduced easily by change of variables from those of the MP laws. Through some elementary calculation, we can actually get that for $z\in[1+\eta,2]$,
\begin{eqnarray*}
&&s_{e1}(\omega)=\frac{(1-z)(c_2-c_1)-\omega+\sqrt{(\omega-(1-z)(c_1+c_2))^2-4(1-z)^2c_1c_2}}{2(1-z)c_1\omega},\\
&&s_{h1}(\omega)=\frac{z(1-c_1-c_2)+\omega-\sqrt{(\omega+z(1+c_1-c_2))^2-4c_1(1-c_2)z^2}}{2z c_1 \omega},\\
&&s_{e2}(\omega)=-\omega^{-1}-\frac{c_1-c_2+z \omega^{-1}-\sqrt{(z \omega^{-1}+c_1+c_2)^2-4c_1c_2}}{2c_2\omega},\\
&&s_{h2}(\omega)=-\omega^{-1}-\frac{1-c_1-c_2-(1-z)\omega^{-1}+\sqrt{((1-z)\omega^{-1}-(1-c_1+c_2))^2-4(1-c_1)c_2}}{2c_2\omega}.
\end{eqnarray*} 
Then by using (\ref{061802}) and (\ref{061803}), we can get through elementary calculation that
\begin{eqnarray*}
R_{e1}(\omega)=\frac{(1-z)c_2}{1-(1-z)c_1\omega},\quad R_{h1}(\omega)=-\frac{z(1-c_2)}{1+z c_1 \omega}.
\end{eqnarray*}
Denote $m_1:=s_1(0)$. It follows from (\ref{061802})-(\ref{061803}) that
\begin{eqnarray*}
\frac{(1-z)c_2}{1+(1-z)c_1m_1}-\frac{z(1-c_2)}{1-z c_1 m_1}-\frac{1}{m_1}=0,
\end{eqnarray*}
which implies that
\begin{eqnarray*}
z(1-z)(c_1^2-c_1)m_1^2+(c_2-c_1+2z c_1-z)m_1-1=0.
\end{eqnarray*}
Solving the above equation we get
\begin{eqnarray*}
m_1=\frac{-(c_2-c_1+2z c_1-z)- \sqrt{z^2+(4c_1c_2-2c_1-2c_2)z+(c_1-c_2)^2}}{2z(1-z)(c_1^2-c_1)},
\end{eqnarray*}
where the minus sign for the square root term is chosen according to the fact that when $z\to 1$, the limit of $m_1$ should exist, since $\mathring{J}\to -\mathring{H}$.
In a similar vein, by using (\ref{061802}) and (\ref{061803}) again, we obtain
\begin{eqnarray*}
&&R_{e2}(\omega)=\frac{c_1-c_2-\sqrt{(c_1-c_2)^2+4z c_2 \omega}}{2c_2 \omega},\nonumber\\
&&R_{h2}(\omega)=\frac{1-c_1-c_2-\sqrt{(1-c_1-c_2)^2+4(z-1) c_2 \omega}}{2c_2 \omega}.
\end{eqnarray*}
The minus signs in the square root terms above are chosen according to the fact (\ref{062001}). Denoting by $m_2:=s_2(0)$, it follows from (\ref{061802})-(\ref{061803}) that
\begin{eqnarray*}
\frac{c_1-c_2-\sqrt{(c_1-c_2)^2-4z c_2 m_2}}{2c_2 m_2}+\frac{1-c_1-c_2-\sqrt{(1-c_1-c_2)^2-4(z-1) c_2m_2 }}{2c_2 m_2}+\frac{1}{m_2}=0,
\end{eqnarray*}
which implies after elementary calculation that
\begin{eqnarray*}
m_2=\frac{c_1+c_2-2c_1c_2-z+\sqrt{z^2+(4c_1c_2-2c_1-2c_2)z+(c_1-c_2)^2}}{2c_2}.
\end{eqnarray*}
Here the plus sign for the square root term can be confirmed by the extreme case $z\to 1$, which implies $\mathring{\mathcal{J}}\to -\mathring{\mathcal{E}}^{-1}$. Now, setting
\begin{eqnarray*}
f(z)=(1-z)m_1,\quad \varrho(z)=m_2,
\end{eqnarray*}
we can conclude the proof.
\section{Appendix}
In this appendix, we provide the proofs of  Lemmas \ref{lem.071001}, \ref{lem.071604} and \ref{lem.071101}.
\subsection{Proof of Lemma \ref{lem.071001}} At first, we prove (i) and (ii) together. Note that
\begin{eqnarray}
\Phi^\varepsilon-\Phi=\varepsilon z \Phi\Phi^\varepsilon,\quad \Psi^\varepsilon-\Psi=\varepsilon  z \Psi\Psi^\varepsilon.   \label{071120}
\end{eqnarray}
Firstly , we show that $||\Phi||$ and $||\Psi||$ are bounded with high probability on $\mathcal{D}_2$ in the sense of (i). We can write
\begin{eqnarray}
&&\Phi(z)=(E-zS_{\mathbf{ww}})^{-1}=S_{\mathbf{ww}}^{-1/2}(S_{\mathbf{ww}}^{-1/2}ES_{\mathbf{ww}}^{-1/2}-zI_p)^{-1}S_{\mathbf{ww}}^{-1/2}, \nonumber\\
&&\Psi(z)=\mathcal{H}(\mathcal{E}-zS_{\mathbf{yy}})^{-1}\mathcal{E}=\mathcal{H}S_{\mathbf{yy}}^{-1/2}(S_{\mathbf{yy}}^{-1/2}\mathcal{E}S_{\mathbf{yy}}^{-1/2}-zI_q)^{-1}S_{\mathbf{yy}}^{-1/2}\mathcal{E}. \label{070701}
\end{eqnarray}
It is elementary to see that
\begin{eqnarray*}
&&||(S_{\mathbf{ww}}^{-1/2}ES_{\mathbf{ww}}^{-1/2}-zI_p)^{-1}||\leq ||(S_{\mathbf{ww}}^{-1/2}ES_{\mathbf{ww}}^{-1/2}-\Re zI_p)^{-1}||,\nonumber\\
&& ||(S_{\mathbf{yy}}^{-1/2}\mathcal{E}S_{\mathbf{yy}}^{-1/2}-zI_q)^{-1}||\leq ||(S_{\mathbf{yy}}^{-1/2}\mathcal{E}S_{\mathbf{yy}}^{-1/2}-\Re zI_q)^{-1}||.
\end{eqnarray*}
According to (\ref{071502}) and the definitions of $E$ and $\mathcal{E}$, we can easily see that
\begin{eqnarray*}
||S_{\mathbf{ww}}^{-1/2}ES_{\mathbf{ww}}^{-1/2}||=||S_{\mathbf{yy}}^{-1/2}\mathcal{E}S_{\mathbf{yy}}^{-1/2}||\leq d_r+\eta
\end{eqnarray*}
holds with overwhelming probability. Considering $\Re z\geq d_r+\frac{3}{2}\eta$ for $z\in \mathcal{D}_2$, we conclude that 
\begin{eqnarray*}
\sup_{z\in\mathcal{D}_2}||\Phi(z)||\leq C\eta^{-1},\quad \sup_{z\in\mathcal{D}_2}||\Psi(z)||\leq C\eta^{-1}
\end{eqnarray*}
hold with overwhelming probability for some positive constant $C$. 
%
According to (\ref{071120}), we see that for some positive constant $C$,
\begin{eqnarray*}
&||\Phi^\varepsilon||-||\Phi||\leq ||\Phi^\varepsilon-\Phi||\leq C\varepsilon||\Phi||\cdot||\Phi^\varepsilon||,&\nonumber\\ \nonumber\\
&||\Psi^\varepsilon||-||\Psi||\leq ||\Psi^\varepsilon-\Psi||\leq C\varepsilon||\Psi||\cdot ||\Psi^\varepsilon||.&
\end{eqnarray*}
on $\mathcal{D}_2$.
By the discussion above, we see that for some positive constants $C_1$ and $C_2$
\begin{eqnarray*}
&||\Phi^\varepsilon||-||\Phi||\leq C_1 \varepsilon ||\Phi^\varepsilon||,&\nonumber\\ \nonumber\\
&||\Psi^\varepsilon||-||\Psi||\leq C_2\varepsilon ||\Psi^\varepsilon||&
\end{eqnarray*}
hold uniformly on $\mathcal{D}_2$, with overwhelming probability. Choosing $\varepsilon$ small enough such that $C_1\varepsilon, C_2\varepsilon\leq 1/2$, then we can immediately get that
\begin{eqnarray*}
||\Phi^\varepsilon||\leq 2||\Phi||,\quad ||\Psi^\varepsilon||\leq 2||\Psi||
\end{eqnarray*}
hold uniformly on $\mathcal{D}_2$ with overwhelming probability.
Hence, we complete the proof of (i) and (ii).

Now, we provide the deterministic bounds for $||\Phi^\varepsilon||$ and $||\Psi^\varepsilon||$ on $\mathcal{D}_2\setminus\mathcal{D}_3$, i.e., prove (iii). 
At first, analogous to $\mathring{J}$ and $\mathring{\mathcal{J}}$ in (\ref{071130}), we denote 
\begin{eqnarray*}
J\equiv J_n(z):=(1-z)E-zH, \quad \mathcal{J}\equiv \mathcal{J}_n(z):= (1-z)\mathcal{H}^{-1}-z\mathcal{E}^{-1}.
\end{eqnarray*} 
Then, we decompose $J-z\varepsilon I_p$ and $\mathcal{J}-z\varepsilon I_q$ into their real parts and imaginary parts as follows,
\begin{eqnarray*}
J-z\varepsilon I_p&=&[(1-\Re z) E-\Re z H-\Re z \varepsilon I_p]-\sqrt{-1}\Im z[ E+ H + \varepsilon I_p]\nonumber\\
&:=&\mathfrak{R}_1^\varepsilon-\sqrt{-1}\Im z \mathfrak{I}^\varepsilon_1\nonumber\\\nonumber\\
\mathcal{J}-z\varepsilon I_q&=&[(1-\Re z)\mathcal{H}^{-1}-\Re z\mathcal{E}^{-1}-\Re z\varepsilon I_q]-\sqrt{-1}\Im z[\mathcal{H}^{-1}+\mathcal{E}^{-1}+\varepsilon I_q]\nonumber\\
&:=&\mathfrak{R}_2^\varepsilon-\sqrt{-1}\Im z \mathfrak{I}_2^\varepsilon.
\end{eqnarray*}
We further split $\mathcal{D}_2\setminus \mathcal{D}_3$ into two parts,
\begin{eqnarray*}
\mathcal{D}_4=\{z\in\mathcal{D}_2\setminus \mathcal{D}_3: |\Im z|\geq \varepsilon\},\quad \mathcal{D}_5=\{z\in\mathcal{D}_2\setminus \mathcal{D}_3: |\Im z|<\varepsilon\}.
\end{eqnarray*}
Now, note that both $\mathfrak{I}_1^\varepsilon$ and $\mathfrak{I}_2^\varepsilon$ are positive-definte, thus
\begin{eqnarray*}
&&\Phi^\varepsilon=(J-z\varepsilon I_p)^{-1}= (\mathfrak{I}_1^\varepsilon)^{-1/2}((\mathfrak{I}_1^\varepsilon)^{-1/2} \mathfrak{R}_1^\varepsilon (\mathfrak{I}_1^\varepsilon)^{-1/2}
-\sqrt{-1}\Im z I_p)^{-1} (\mathfrak{I}_1^\varepsilon)^{-1/2},\nonumber\\
&&\Psi^\varepsilon=(\mathcal{J}-z\varepsilon I_q)^{-1}=(\mathfrak{I}_2^\varepsilon)^{-1/2}((\mathfrak{I}_2^\varepsilon)^{-1/2} \mathfrak{R}_2^\varepsilon (\mathfrak{I}_2^\varepsilon)^{-1/2}
-\sqrt{-1}\Im z I_q)^{-1} (\mathfrak{I}_2^\varepsilon)^{-1/2}.
\end{eqnarray*}
Note that
\begin{eqnarray*}
\mathfrak{R}_1^\varepsilon=E-\Re z \mathfrak{I}_1^\varepsilon,\quad \mathfrak{R}_2^\varepsilon= \mathcal{H}^{-1}-\Re z I_2^\varepsilon.
\end{eqnarray*}
Hence, by the obvious facts that $(\mathfrak{I}_1^\varepsilon)^{-1/2} E (\mathfrak{I}_1^\varepsilon)^{-1/2}\preceq I_p$ and $(\mathfrak{I}_2^\varepsilon)^{-1/2} \mathcal{H}^{-1} (\mathfrak{I}_2^\varepsilon)^{-1/2}\preceq I_q$, we see that
\begin{eqnarray*}
&&(\mathfrak{I}_1^\varepsilon)^{-1/2} \mathfrak{R}_1^\varepsilon (\mathfrak{I}_1^\varepsilon)^{-1/2}=(\mathfrak{I}_1^\varepsilon)^{-1/2} E (\mathfrak{I}_1^\varepsilon)^{-1/2}-\Re z I_p \preceq (1-\Re z)I_p,\nonumber\\
&&(\mathfrak{I}_2^\varepsilon)^{-1/2} \mathfrak{R}_2^\varepsilon (\mathfrak{I}_2^\varepsilon)^{-1/2}=(\mathfrak{I}_2^\varepsilon)^{-1/2} \mathcal{H}^{-1} (\mathfrak{I}_2^\varepsilon)^{-1/2}-\Re z I_q \preceq (1-\Re z)I_q.
\end{eqnarray*}
Now note that
\begin{eqnarray*}
||((\mathfrak{I}_1^\varepsilon)^{-1/2} \mathfrak{R}_1^\varepsilon (\mathfrak{I}_1^\varepsilon)^{-1/2}
-\sqrt{-1}\Im z I_p)^{-1}||&\leq &\min \{||(\mathfrak{I}_1^\varepsilon)^{-1/2} \mathfrak{R}_1^\varepsilon (\mathfrak{I}_1^\varepsilon)^{-1/2}||^{-1}, |\Im z|^{-1}\}\nonumber\\
&\leq &  \min \{\eta^{-1},\varepsilon^{-1}\}.
\end{eqnarray*}
The first bound is on $\mathcal{D}_5$ and the second is on $\mathcal{D}_4$. Moreover, we see that
\begin{eqnarray*}
||(\mathfrak{I}_1^\varepsilon)^{-1/2}||\leq \varepsilon^{-1/2}.
\end{eqnarray*}
Consequently, 
\begin{eqnarray*}
\sup_{z\in\mathcal{D}_2\setminus\mathcal{D}_3}||\Phi^\varepsilon||\leq \min \{\eta^{-1}\varepsilon^{-1},\varepsilon^{-2}\}.
\end{eqnarray*}
Analogously, we also have
\begin{eqnarray*}
\sup_{z\in\mathcal{D}_2\setminus\mathcal{D}_3}||\Psi^\varepsilon||\leq \min \{\eta^{-1}\varepsilon^{-1},\varepsilon^{-2}\}.
\end{eqnarray*}
Finally, for (iv),  we can actually reproduce the discussion for (iii) verbatim, as long as we replace $J$ and $\mathcal{J}$ by $\mathring{J}$ and $\mathring{\mathcal{J}}$ respectively. Specifically, we also decompose $\mathring{J}$ and $\mathring{\mathcal{J}}$ into their real parts and imaginary parts as follows,
\begin{eqnarray*}
\mathring{J}&=&[(1-\Re z) \mathring{E}-\Re z \mathring{H}]-\sqrt{-1}\Im z[ \mathring{E}+\mathring{H} ]:=\mathring{\mathfrak{R}}_1-\sqrt{-1}\Im z \mathring{\mathfrak{I}}_1\nonumber\\\nonumber\\
\mathring{\mathcal{J}}&=&[(1-\Re z)\mathring{\mathcal{H}}^{-1}-\Re z\mathring{\mathcal{E}}^{-1}]-\sqrt{-1}\Im z[\mathring{\mathcal{H}}^{-1}+\mathring{\mathcal{E}}^{-1}]:=\mathring{\mathfrak{R}}_2-\sqrt{-1}\Im z \mathring{\mathfrak{I}}_2.
\end{eqnarray*}
By construction, it is easy to see that both $||\mathring{\mathfrak{I}}_1^{-1}||$ and $||\mathring{\mathfrak{I}}_2^{-1}||$ are bounded. The remaining discussion is nearly the same as that for (iii). We omit the details.
Hence, we conclude the proof.

\subsection{Proof of Lemma \ref{lem.071604}}
At first, it is well known  the statements in Lemma \ref{lem.071604} hold if we replace $\boldsymbol{\upsilon}_i$ and $\boldsymbol{\upsilon}_j$ by two i.i.d. $N(\mathbf{0}_{n\times 1}, n^{-1}I_n)$-distributed vectors $\mathbf{g}_i$ and $\mathbf{g}_j$.  It is also well known that some (but not all) columns of $\mathcal{U}_n$ can be simultaneously approximated by i.i.d. $N(\mathbf{0}_{n\times 1}, n^{-1}I_n)$-distributed vectors. To be specific, we cite the following result, courtesy of Jiang \cite{Jiang2005}.
\begin{lem} \label{lem.07260001} Let $u_{\alpha\beta}$ be the $(\alpha,\beta)$th entry of $\mathcal{U}_n$. There exist $g_{\alpha\beta}, \alpha,\beta=1,\ldots, n$, which are i.i.d. $N(0,n^{-1})$ variables, such that for $m=1,\ldots, n$, 
\begin{eqnarray*}
\mathbb{P}(\max_{\substack{1\leq \alpha\leq n \\1\leq \beta\leq m}}\sqrt{n}|u_{\alpha\beta}-g_{\alpha\beta}|\geq rs+2t)\leq 4 me^{-nr^2/16}+3mn\left(s^{-1}e^{-s^2/2}+t^{-1}\left(1+\frac{t^2}{3(m+\sqrt{n})}\right)^{-n/2}\right)
\end{eqnarray*}
for any $r\in(0,1/4)$, $s>0$, $t>0$ and $m\leq (r/2) n$.
\end{lem}
Due to symmetry, obviously, $\beta$ can range over any subset of $\{1,\ldots, n\}$ of cardinality $m$ in the above lemma. Now we choose,
\begin{eqnarray*}
r=n^{-1/2}\log n,\quad s=\log n, \quad t=n^{-1/2} (\log n)^2, \quad m=2,
\end{eqnarray*}
we can roughly but safely say that
\begin{eqnarray}
\max_{1\leq \alpha\leq n, \beta=i,j}|u_{\alpha\beta}-g_{\alpha\beta}|\leq n^{-1+\varepsilon} \label{071910}
\end{eqnarray}
with overwhelming probability for any small fixed $\varepsilon>0$. By definition, we know that $\boldsymbol{\upsilon}_\beta=(u_{1\beta},\ldots, u_{n \beta})'$ for $\beta=i,j$. Analogously, we let $\mathbf{g}_\beta=(g_{1\beta},\ldots g_{n \beta})$, with $\beta=i,j$ and $g_{\alpha\beta}$'s are those Gaussian variables in Lemma \ref{lem.07260001}. A direct consequence of (\ref{071910}) is that for any given $\varepsilon>0$, 
\begin{eqnarray}
||\boldsymbol{\upsilon}_\beta-\mathbf{g}_\beta||\leq n^{-1/2+\varepsilon},\quad \beta=i,j\label{071911}
\end{eqnarray}
holds with overwhelming probability.
 Now, observe that 
\begin{eqnarray*}
\boldsymbol{\upsilon}_i'A \boldsymbol{\upsilon}_i=\mathbf{g}'_i A\mathbf{g}_i+ (\boldsymbol{\upsilon}_i'-\mathbf{g}'_i) A\mathbf{g}_i+\mathbf{g}'_i A(\boldsymbol{\upsilon}_i-\mathbf{g}_i)+(\boldsymbol{\upsilon}_i'-\mathbf{g}'_i) A(\boldsymbol{\upsilon}_i-\mathbf{g}_i),
\end{eqnarray*}
from which we have
\begin{eqnarray}
||\boldsymbol{\upsilon}_i'A \boldsymbol{\upsilon}_i-\mathbf{g}'_i A\mathbf{g}_i||\leq 2||A||\cdot||\mathbf{g}_i||\cdot ||\boldsymbol{\upsilon}_i-\mathbf{g}_i||+||A||\cdot||\boldsymbol{\upsilon}_i-\mathbf{g}_i||^2\leq \varepsilon \label{0726200}
\end{eqnarray}
with overwhelming probability for any small constant $\varepsilon>0$, by using (\ref{071911}), together with the assumption that $||A||\leq C$ and the elementary fact that $||\mathbf{g}_i||$ is bounded with overwhelming probability. In a similar manner, on can get
\begin{eqnarray}
||\boldsymbol{\upsilon}_i'A \boldsymbol{\upsilon}_j-\mathbf{g}'_i A\mathbf{g}_j||\leq \varepsilon, \quad ||\boldsymbol{\upsilon}_i'\mathbf{b}-\mathbf{g}_i'\mathbf{b}||\leq \varepsilon \label{0726300}
\end{eqnarray}
with overwhelming probability for any small constant $\varepsilon>0$. Then, as mentioned above, (i)-(iii) do hold if we replace $\boldsymbol{\upsilon}_i$ and $\boldsymbol{\upsilon}_j$ by $\mathbf{g}_i$ and $\mathbf{g}_j$ respectively, hence, according to (\ref{0726200}) and (\ref{0726300}), we can get the conclusion by slightly adjusting the value of $\varepsilon$.

\subsection{Proof of Lemma \ref{lem.071101}}
To provide the concentration for the normalized trace of $\Phi^\varepsilon$ and $\Psi^\varepsilon$, we will rely on the commonly used martingale difference strategy. In the sequel, we will focus on the region $\mathcal{D}_4$, the uniform bound can be extended to the whole $\mathcal{D}_2\setminus\mathcal{D}_3=\mathcal{D}_4\cup\mathcal{D}_5$ via an analogous argument as (\ref{071143}).

Note that both $E$ and $H$ are scaled Wishart matrices. Hence, we can write $E=n^{-1}\sum_{i=1}^q \boldsymbol{\epsilon}_i \boldsymbol{\epsilon}_i'$ and $H=n^{-1}\sum_{i=1}^{n-q}\mathbf{h}_i\mathbf{h}_i'$, where $\boldsymbol{\epsilon}_1,\ldots, \boldsymbol{\epsilon}_q, \mathbf{h}_1,\ldots, \mathbf{h}_{n-q}$ are i.i.d. standard $p$-dimensional normal vectors.  For simplicity, we further set
\begin{eqnarray*}
\boldsymbol{\eta}_{i}=\left\{
\begin{array}{ccc}
\boldsymbol{\epsilon}_i, &\text{for }\quad i=1,\ldots, q\\\\
\mathbf{h}_{i-q}, &\qquad \text{for}\quad i=q+1,\ldots,n.
\end{array}
\right.
\end{eqnarray*}
In addition, let $\Phi_{(k)}^\varepsilon$ be the matrix obtained from $\Phi^\varepsilon$ via replacing $\boldsymbol{\eta}_k$ by $\mathbf{0}_{p\times 1}$. Correspondingly, we define $E_{(k)}$ (resp. $H_{(k)}$) to be the matrix obtained from $E$ (resp. $H$) via replacing $\boldsymbol{\eta}_i$ (resp. $\mathbf{h}_k$) by $\mathbf{0}_{p\times 1}$. Hence, for example, we have the relation $\Phi_{(k)}^\varepsilon=((1-z)E-zH_{(k-p)}-z\varepsilon I_p)^{-1}$ in case $k\geq p+1$. Now, let $\mathbb{E}_k$  be the operator of taking expectation with respect to $\{\boldsymbol{\eta}_{i}\}_{i=1}^k$ and make the convention that $\mathbb{E}_0$ is the identity operator. Then we can write
\begin{eqnarray}
\frac{1}{p}\text{tr}\Phi^\varepsilon(z)-\mathbb{E}\frac{1}{p}\text{tr}\Phi^\varepsilon(z)&=&\frac{1}{p}\sum_{i=1}^{n} (\mathbb{E}_{i-1}-\mathbb{E}_i) \text{tr}\Phi^\varepsilon(z)\nonumber\\
&=&\frac{1}{p}\sum_{i=1}^{n} (\mathbb{E}_{i-1}-\mathbb{E}_i) (\text{tr}\Phi^\varepsilon(z)-\text{tr}\Phi_{(i)}^\varepsilon(z)) \label{071144}
\end{eqnarray}
Now we calculate the one step difference. We just state the details for the case of $i\in\{1,\ldots,q\}$, the other case, $i\in\{q+1,\ldots,2q\}$, is just analogous. We start from the fact
\begin{eqnarray}
|\text{tr}\Phi^\varepsilon(z)-\text{tr}\Phi_{(i)}^\varepsilon(z)|&=&\left|\frac{n^{-1}(1-z)\boldsymbol{\eta}_i'(\Phi_{(i)}^\varepsilon(z))^2\boldsymbol{\eta}_i}{1+n^{-1}(1-z)\boldsymbol{\eta}_i'\Phi_{(i)}^\varepsilon(z)\boldsymbol{\eta}_i}\right|
\nonumber\\
&\leq&  \frac{|\boldsymbol{\eta}_i'(\Phi_{(i)}^\varepsilon(z))^2\boldsymbol{\eta}_i|}{|n\Im z/|1-z|^2+\Im (\boldsymbol{\eta}_i'\Phi_{(i)}^\varepsilon(z)\boldsymbol{\eta}_i)|}.
\label{0726400}
\end{eqnarray}
Now, we show that
\begin{eqnarray}
|\boldsymbol{\eta}_i'(\Phi_{(i)}^\varepsilon(z))^2\boldsymbol{\eta}_i|\leq |\boldsymbol{\eta}_i'\Phi_{(i)}^\varepsilon(z)\Phi_{(i)}^\varepsilon(\bar z)\boldsymbol{\eta}_i|. \label{071141}
\end{eqnarray}
To see this, we denote $\Phi_{(i)}^\varepsilon(z):=A+\sqrt{-1}B$, where $A$ and $B$ are the real and imaginary parts of $\Phi_{(i)}$ respectively. Denoting by $\mathfrak{R}_{1,(i)}^\varepsilon$ (resp. $\mathfrak{I}^\varepsilon_{1,(i)}$) the matrix obtained from $\mathfrak{R}^\varepsilon_1$ (resp. $\mathfrak{I}^\varepsilon_1$) via replacing $\boldsymbol{\eta}_i$ by $\mathbf{0}_{p\times 1}$, we have
\begin{eqnarray*}
\Phi_{(i)}^\varepsilon(z)=(\mathfrak{R}_{1,(i)}^\varepsilon-\sqrt{-1}\Im z \mathfrak{I}^\varepsilon_{1,(i)})^{-1}=(\mathfrak{I}^\varepsilon_{1,(i)})^{-1/2}\mathcal{I}^\varepsilon_{1,(i)}(z) (\mathfrak{I}^\varepsilon_{1,(i)}))^{-1/2},
\end{eqnarray*}
where we used the notation
\[\mathcal{I}^\varepsilon_{1,(i)}(z):=((\mathfrak{I}^\varepsilon_{1,(i)})^{-1/2} \mathfrak{R}_{1,(i)}^\varepsilon (\mathfrak{I}^\varepsilon_{1,(i)})^{-1/2}
-\sqrt{-1}\Im z I_p)^{-1}.\]
 Then, it is not difficult to see from the spectral decomposition of  $\mathcal{I}^\varepsilon_{1,(i)}(z)$ that both $A$ and $B$ are symmetric, and $\Phi_{(i)}^\varepsilon(\bar{z})=A-\sqrt{-1}B$. 
 Now, briefly writing $\boldsymbol{\eta}_i$ as $\boldsymbol{\eta}$, we need to compare
\begin{eqnarray*}
\text{l.h.s. of (\ref{071141})}=|\boldsymbol{\eta}'(A+\sqrt{-1}B)^2\boldsymbol{\eta}|^2=(\boldsymbol{\eta}'A^2\boldsymbol{\eta}-\boldsymbol{\eta}' B^2\boldsymbol{\eta}')^2+(\boldsymbol{\eta}'AB\boldsymbol{\eta}+\boldsymbol{\eta}'BA\boldsymbol{\eta})^2
\end{eqnarray*}
and
\begin{eqnarray*}
\text{r.h.s. of (\ref{071141})}=|\boldsymbol{\eta}'(A+\sqrt{-1}B)(A-\sqrt{-1}B)\boldsymbol{\eta}|^2=(\boldsymbol{\eta}'A^2\boldsymbol{\eta}+\boldsymbol{\eta}' B^2\boldsymbol{\eta}')^2+(\boldsymbol{\eta}'AB\boldsymbol{\eta}-\boldsymbol{\eta}'BA\boldsymbol{\eta})^2.
\end{eqnarray*}
Then, it is direct to see 
\begin{eqnarray*}
\text{r.h.s. of (\ref{071141})}-\text{l.h.s. of (\ref{071141})}=4(\boldsymbol{\eta}'A^2\boldsymbol{\eta}\cdot \boldsymbol{\eta}' B^2\boldsymbol{\eta}'-\boldsymbol{\eta}'AB\boldsymbol{\eta}\cdot \boldsymbol{\eta}'BA\boldsymbol{\eta})\geq 0
\end{eqnarray*}
by Cauchy-Schwarz inequality. Hence, we deduce from (\ref{0726400}) and (\ref{071141}) that
\begin{eqnarray}
|\text{tr}\Phi^\varepsilon(z)-\text{tr}\Phi_{(i)}^\varepsilon(z)|\leq
 \frac{|\boldsymbol{\eta}_i'\Phi_{(i)}^\varepsilon(z)\Phi_{(i)}^\varepsilon(\bar z)\boldsymbol{\eta}_i|}{|n\Im z/|1-z|^2+\Im (\boldsymbol{\eta}_i'\Phi_{(i)}^\varepsilon(z)\boldsymbol{\eta}_i)|}\leq \frac{|\boldsymbol{\eta}_i'\Phi_{(i)}^\varepsilon(z)\Phi_{(i)}^\varepsilon(\bar z)\boldsymbol{\eta}_i|}{|\Im (\boldsymbol{\eta}_i'\Phi_{(i)}^\varepsilon(z)\boldsymbol{\eta}_i)|}. \label{070804}
\end{eqnarray}
Here, in the last step we used the fact that $\Im (\boldsymbol{\eta}_i'\Phi_{(i)}^\varepsilon(z)\boldsymbol{\eta}_i)$ share the same sign with $\Im z$. 
Now, we arrive at the stage to bound the r.h.s. of (\ref{070804}).  Note that $\mathcal{I}^\varepsilon_{1,(i)}(\bar{z})=\overline{\mathcal{I}^\varepsilon_{1,(i)}(z)}$, thus
\begin{eqnarray*}
&&|\boldsymbol{\eta}_i'\Phi_{(i)}^\varepsilon(z)\Phi_{(i)}^\varepsilon(\bar z)\boldsymbol{\eta}_i|= \boldsymbol{\eta}_i'(\mathfrak{I}^\varepsilon_{1,(i)})^{-1/2}(\mathcal{I}_{1,(i)}^\varepsilon(z))^{-1} (\mathfrak{I}^\varepsilon_{1,(i)})^{-1}((\mathcal{I}_{1,(i)}^\varepsilon(\bar{z}))^{-1}
(\mathfrak{I}^\varepsilon_{1,(i)})^{-1/2}\boldsymbol{\eta}_i\nonumber\\
&& \leq ||(\mathfrak{I}^\varepsilon_{1,(i)})^{-1}||\cdot \boldsymbol{\eta}_i'(\mathfrak{I}^\varepsilon_{1,(i)})^{-1/2}|\mathcal{I}_{1,(i)}^\varepsilon(z)|^{-2} (\mathfrak{I}^\varepsilon_{1,(i)})^{-1/2}\boldsymbol{\eta}_i\nonumber\\
&& =| z|^{-1}||(\mathfrak{I}^\varepsilon_{1,(i)})^{-1}||\cdot \boldsymbol{\eta}_i'(\mathfrak{I}^\varepsilon_{1,(i)})^{-1/2}\Im (\mathcal{I}_{1,(i)}^\varepsilon(z)) (\mathfrak{I}^\varepsilon_{1,(i)})^{-1/2}\boldsymbol{\eta}_i.
\end{eqnarray*}
Moreover, obviously, we have
\begin{eqnarray*}
\Im (\boldsymbol{\eta}_i'\Phi_{(i)}^\varepsilon(z)\boldsymbol{\eta}_i)=\boldsymbol{\eta}_i'(\mathfrak{I}^\varepsilon_{1,(i)})^{-1/2}\Im (\mathcal{I}_{1,(i)}^\varepsilon(z)) (\mathfrak{I}^\varepsilon_{1,(i)})^{-1/2}\boldsymbol{\eta}_i.
\end{eqnarray*}
Consequently, we have
\begin{eqnarray*}
\sup_{z\in\mathcal{D}_4}|\text{tr}\Phi^\varepsilon(z)-\text{tr}\Phi_{(i)}^\varepsilon(z)|\leq \sup_{z\in\mathcal{D}_4}|\Im z|^{-1}||(\mathfrak{I}^\varepsilon_{1,(i)})^{-1}||\leq \varepsilon^{-2}.
\end{eqnarray*}
So on $\mathcal{D}_4$, we can bound the one step difference as above. Now, we go back to (\ref{071144}), and use a complex version of Burkholder's inequality for martingale difference sequence (see Lemma 2.11 of \cite{BS2009}), then we can get that for any positive number $\ell\geq 1/2$, 
\begin{eqnarray*}
\mathbb{E}|p^{-1}\text{tr}\Phi^\varepsilon(z)-p^{-1}\mathbb{E}\text{tr}\Phi^\varepsilon(z)|^{2\ell}\leq C_\ell p^{-2\ell}\mathbb{E}(\sum_{i=1}^{n} |(\mathbb{E}_{i-1}-\mathbb{E}_i) (\text{tr}\Phi^\varepsilon(z)-\text{tr}\Phi_{(i)}^\varepsilon(z))|^2)^{\ell}\leq C'_\ell\varepsilon^{-4\ell} p^{-\ell}
\end{eqnarray*}
for some positive constants $C_\ell$ and $C'_\ell$, in case $z\in \mathcal{D}_4$. Then Markov's inequality implies
\begin{eqnarray*}
\mathbb{P}(|p^{-1}\text{tr}\Phi^\varepsilon(z)-p^{-1}\mathbb{E}\text{tr}\Phi^\varepsilon(z)|\geq \varepsilon)\leq C'_\ell\varepsilon^{-6\ell} p^{-\ell}
\end{eqnarray*}
for any $\ell>\frac 12$ and  $z\in \mathcal{D}_4$. The uniform bound on $\mathcal{D}_4$ can be obtained via a standard lattice argument as that for (\ref{072504}), and the further extension to the whole $\mathcal{D}_2\setminus\mathcal{D}_3$ can be derived by an discussion similar to (\ref{071143}), as mentioned above. We just leave the details to the reader.

Now, we turn to $\Psi^\varepsilon$.  Again, we use the martingale difference argument. To this end, we write $\mathcal{E}=n^{-1}\sum_{i=1}^p\tilde{\boldsymbol{\varepsilon}}_i\tilde{\boldsymbol{\varepsilon}}_i$ and $\mathcal{H}=n^{-1}\sum_{i=1}^{n-p}\tilde{\mathbf{h}}_i\tilde{\mathbf{h}}_i'$, where $\tilde{\boldsymbol{\epsilon}}_1,\ldots, \tilde{\boldsymbol{\epsilon}}_p, \tilde{\mathbf{h}}_1,\ldots, \tilde{\mathbf{h}}_{n-p}$ are i.i.d. standard $q$-dimensional normal vectors. And for simplicity, we set
\begin{eqnarray*}
\tilde{\boldsymbol{\eta}}_i=\left\{
\begin{array}{ccc}
\tilde{\boldsymbol{\epsilon}}_i, &\text{for }\quad i=1,\ldots, p\\\\
\tilde{\mathbf{h}}_{i-p}, &\qquad \text{for}\quad i=p+1,\ldots,n.
\end{array}
\right.
\end{eqnarray*}
Similarly, let $\Psi_{(k)}^\varepsilon$ be the matrix obtain from $\Psi^\varepsilon$ via replacing $\tilde{\boldsymbol{\eta}}_k$ by $\mathbf{0}_{q\times 1}$. Analogously, we can define $\tilde{H}_{(k)}$ and $\tilde{E}_{(k)}$. Now, let $\tilde{\mathbb{E}}_k$  be the operator of taking expectation with respect to $\{\tilde{\boldsymbol{\eta}}_{i}\}_{i=1}^k$ and make the convention that $\tilde{\mathbb{E}}_0$ is the identity operator. Then we can write
\begin{eqnarray*}
\frac{1}{q}\text{tr}\Psi^\varepsilon(z)-\mathbb{E}\frac{1}{q}\text{tr}\Psi^\varepsilon(z)&=&\frac{1}{q}\sum_{i=1}^{n} (\tilde{\mathbb{E}}_{i-1}-\tilde{\mathbb{E}}_i) \text{tr}\Psi^\varepsilon(z)\nonumber\\
&=&\frac{1}{q}\sum_{i=1}^{n} (\tilde{\mathbb{E}}_{i-1}-\tilde{\mathbb{E}}_i) (\text{tr}\Psi^\varepsilon(z)-\text{tr}\Psi_{(i)}^\varepsilon(z))
\end{eqnarray*}
Similar to the discussion on $\Phi^\varepsilon$, it suffices to bound the one step difference. 
For simplicity, we only state the estimation for the case of $i\in\{1,\ldots,p\}$. Note that
\begin{eqnarray}
|\text{tr}\Psi^\varepsilon(z)-\text{tr}\Psi_{(i)}^\varepsilon(z)|&=&|z|\left|\frac{\tilde{\boldsymbol{\eta}}_i' \mathcal{E}_{(i)}^{-1}\Psi^\varepsilon(z)\Psi_{(i)}^\varepsilon(z) \mathcal{E}_{(i)}^{-1}\tilde{\boldsymbol{\eta}}_i}{1+\tilde{\boldsymbol{\eta}}_i' \mathcal{E}_{(i)}^{-1}\tilde{\boldsymbol{\eta}}_i}\right| \label{071401}
\end{eqnarray}
Now, denoting analogously
\begin{eqnarray*}
\Psi_{(i)}^\varepsilon(z)&=&(\mathfrak{R}_{2,(i)}^\varepsilon-\sqrt{-1}\Im z \mathfrak{I}^\varepsilon_{2,(i)})^{-1}\nonumber\\
&=&(\mathfrak{I}^\varepsilon_{2,(i)})^{-1/2}((\mathfrak{I}^\varepsilon_{2,(i)})^{-1/2} \mathfrak{R}_{2,(i)}^\varepsilon (\mathfrak{I}^\varepsilon_{2,(i)})^{-1/2}
-\sqrt{-1}\Im z I_p)^{-1} (\mathfrak{I}^\varepsilon_{2,(i)}))^{-1/2}\nonumber\\
&:=&(\mathfrak{I}^\varepsilon_{2,(i)})^{-1/2}\mathcal{I}^\varepsilon_{2,(i)}(z)(\mathfrak{I}^\varepsilon_{2,(i)})^{-1/2},
\end{eqnarray*}
It is not difficult to see that
\begin{eqnarray}
|\tilde{\boldsymbol{\eta}}_i' \mathcal{E}_{(i)}^{-1}\Psi^\varepsilon(z)\Psi_{(i)}^\varepsilon(z) \mathcal{E}_{(i)}^{-1}\tilde{\boldsymbol{\eta}}_i|\leq  ||\Psi^\varepsilon(z)||\cdot ||\mathcal{I}^\varepsilon_{2,(i)}(z)||\cdot  \tilde{\boldsymbol{\eta}}_i' \mathcal{E}_{(i)}^{-1}(\mathfrak{I}^\varepsilon_{2,(i)})^{-1}\mathcal{E}_{(i)}^{-1}\tilde{\boldsymbol{\eta}}_i. \label{0726500}
\end{eqnarray}
 Note that both $||\Psi^\varepsilon(z)||$ and $||\mathcal{I}^\varepsilon_{2,(i)}(z)||$ are bounded on $\mathcal{D}_2\setminus\mathcal{D}_3$. Moreover, it is obvious that 
\begin{eqnarray*}
\mathcal{E}_{(i)}^{-1/2}(\mathfrak{I}^\varepsilon_{2,(i)})^{-1}\mathcal{E}_{(i)}^{-1/2}\preceq I_q,
\end{eqnarray*}
which together with (\ref{0726500}) implies that 
\begin{eqnarray*}
|\tilde{\boldsymbol{\eta}}_i' \mathcal{E}_{(i)}^{-1}\Psi^\varepsilon(z)\Psi_{(i)}^\varepsilon(z) \mathcal{E}_{(i)}^{-1}\tilde{\boldsymbol{\eta}}_i|\leq C\tilde{\boldsymbol{\eta}}_i' \mathcal{E}_{(i)}^{-1}\tilde{\boldsymbol{\eta}}_i
\end{eqnarray*}
for some positive constant $C$, on $\mathcal{D}_2\setminus\mathcal{D}_3$. Plugging this bound into (\ref{071401}) yields 
\begin{eqnarray*}
|\text{tr}\Psi^\varepsilon(z)-\text{tr}\Psi_{(i)}^\varepsilon(z)|\leq C
\end{eqnarray*}
for some positive constant $C$, on $\mathcal{D}_2\setminus\mathcal{D}_3$. The remaining discussion is similar to that for $\Phi^\varepsilon$, hence omitted.
Therefore, we conclude the proof.

\end{document}